\documentclass[a4paper,12pt]{article}
\usepackage{amsmath,amssymb} \usepackage{euler,eucal}
\usepackage{ifxetex}
\ifxetex
  \usepackage{fontspec} \usepackage{polyglossia}
  \defaultfontfeatures{Mapping=tex-text}
  \setromanfont[Alternate=1]{URW Palladio L}
\else
  \usepackage[utf8]{inputenc}
  \usepackage[T1]{fontenc}
  \usepackage[english]{babel}
  \usepackage{tgpagella}
\fi
\usepackage[unicode=true,plainpages=false]{hyperref}
\usepackage{graphicx}
\hypersetup{colorlinks=false,pdfborder={0 0 0.1},linkcolor=magenta,anchorcolor=magenta,urlcolor=blue,citecolor=blue}

\usepackage[a4paper]{geometry}
\usepackage{algorithmic}
\usepackage[ruled]{algorithm}
\usepackage{amsthm}
\theoremstyle{definition}

\newtheorem{remark}{Remark}

\newtheorem{theorem}{Theorem}

\usepackage{tikz} \usetikzlibrary{positioning,shapes}
\usepackage{pgfplots} \usetikzlibrary{plotmarks}
\def\TIKZcoresize{14mm}
\def\TIKZlspan{8mm}
\def\TIKZhspan{7mm}
\input ./tikz/ttdef
\let\eps=\varepsilon

\let\le=\leqslant

\let\leq=\leqslant
\let\geq=\geqslant

\def\A{A} \def\x{x} \def\z{z} \def\s{s} \def\t{t} \def\y{y} \def\c{c}

\def\i{\mathbf{i}} \def\r{\mathbf{r}}
\def\T{\mathcal{T}}
\def\P{\mathcal{P}}
\def\X{\mathcal{X}} \def\Y{\mathcal{Y}} \def\Z{\mathcal{Z}} \def\G{\mathcal G}
\def\O{\mathcal{O}}
\def\H{\mathtt{H}}
\def\trace{\mathop{\mathrm{Tr}}\nolimits}
\def\rank{\mathop{\mathrm{rank}}\nolimits}
\def\als{\mathop{\mathrm{ALS}}\nolimits}
\def\cond{\mathop{\mathrm{cond}}\nolimits}
\def\grad{\mathop{\mathrm{grad}}\nolimits}
\def\const{\mathop{\mathrm{const}}\nolimits}
\def\Span{\mathop{\mathrm{span}}\nolimits}
\def\trans{*}
\def\lmax{\lambda_{\mathrm{max}}}
\def\lmin{\lambda_{\mathrm{min}}}
\def\dz{\delta{z}}
 \def\C{\mathbb{C}}
\def\new{\mathrm{new}}
\begin{document}
\author{Sergey V. Dolgov\thanks{Max-Planck-Institut f{\"u}r Mathematik in den Naturwissenschaften, Inselstr. 22-26, D-04103 Leipzig, Germany ({\tt sergey.v.dolgov@gmail.com}).}~
  and Dmitry V. Savostyanov\thanks{University of Southampton, Department of Chemistry, Highfield Campus, Southampton SO17 1BJ, United Kingdom ({\tt dmitry.savostyanov@gmail.com})}}
\title{Alternating minimal energy methods for linear systems in higher dimensions. Part I: SPD systems\thanks{Partially supported by
         RFBR grants  12-01-00546-a, 11-01-12137-ofi-m-2011, 11-01-00549-a, 12-01-33013, 12-01-31056,
         Russian Fed. Gov. contracts No. $\Pi 1112$, 14.740.11.0345, 16.740.12.0727 at Institute of Numerical Mathematics, Russian Academy of Sciences,
         and EPSRC grant EP/H003789/1  at the University of Southampton.
         }}
\date{January 25, 2013}
\maketitle

\begin{abstract}
We introduce a family of numerical algorithms for the solution of linear system in higher dimensions with the matrix and right hand side given and the solution sought in the tensor train format.
The proposed methods are rank--adaptive and follow the alternating directions framework, but in contrast to ALS methods, in each iteration a tensor subspace is enlarged by a set of vectors chosen similarly to the steepest descent algorithm.
The convergence is analysed in the presence of   approximation errors and the geometrical convergence rate is estimated and related to the one of the steepest descent.
The complexity of the presented algorithms is linear in the mode size and dimension and the convergence demonstrated in the numerical experiments is comparable to the one of the DMRG--type algorithm.

\end{abstract}


\section{Introduction}
Linear systems arising from high--dimensional problems usually can not be solved by standard numerical algorithms.
If the equation is considered in $d$ dimensions on a $n_1 \times n_2 \times \ldots \times n_d$ grid, the number of unknowns $n_1 \ldots n_d$ scales exponentially with $d,$ and even for moderate dimension $d$ and mode sizes $n_k$ the numerical complexity lays far beyond the technical possibilities of modern workstations and parallel systems.
To make the problem tractable, different approximations are proposed, including sparse grids~\cite{smolyak-1963,griebel-sparsegrids-2004} and tensor product methods~\cite{kolda-review-2009,khor-qtt-2011,khor-survey-2011,hackbusch-2012}.
In this paper we consider the linear system
$
\A \x = \y,
$ 
where the matrix $\A$ and right-hand-side $\y$ are given and approximate solution $\x$ is sought in the \emph{tensor train} (TT) format.
Methods based on the TT format, also known as a \emph{linear tensor network}, are novel and particularly interesting among all tensor product methods due to their robustness and simplicity.

The numerical optimization on tensor networks was first considered in quantum physics community by S.~White~\cite{white-dmrg-1993}, who introduces the \emph{matrix product states} (MPS) formalism to represent the ground state of a spin system together with the \emph{density matrix renormalization group} (DMRG) optimization scheme.
The tensor train format and some computational methods were independently re-discovered in the papers of Oseledets and Tyrtyshnikov (see \cite{osel-tt-2011} and references therein) until the results of White~et.~al. were popularized in the numerical mathematics community by R.~Schneider~\cite{holtz-ALS-DMRG-2012}.
The questions concerning the convergence properties of alternating schemes for different tensor product formats were immediately raised and studied.
The experimental results from quantum physics show the notably fast convergence of DMRG for the ground state problem, i.e., finding the minimal eigenstate of a system, but give no theoretical justification for this observation.
The \emph{alternating least squares} (ALS) algorithm was used in multilinear analysis for the computation of \emph{canonical} tensor decomposition since early results of Hitchcock~\cite{hitchcock-sum-1927} and was known for its monotone but very slow convergence.
For ALS there is also a lack of convergence estimates both in the classical papers~\cite{harshman-parafac-1970,cc-parafac-1970},
and in the recent ones, where ALS was applied to the Tucker model~\cite{mpi-chem-2007,ost-latensor-2009},
tensor trains~\cite{Os-mvk2-2011},
hierarchical Tucker format~\cite{tobler-ht_dmrg-2011}
and high--dimensional interpolation~\cite{ot-ttcross-2010}.

In recent papers by Uschmajew~\cite{ushmaev-cp-2012,ushmaev-tt-2013} the local convergence of ALS is proven for the canonical and tensor train decompositions.
This is a major theoretical breakthrough, which unfortunately does not immediately lead to practical algorithms due to the local character of convergence studied, unjustified assumptions on the structure of the Hessian, and very strong requirements on the accuracy of the initial guess.
The convergence rate of ALS is difficult to estimate partly due to the complex geometrical structure of manifolds defined by tensor networks.
This problem is now approached from several directions, and we might expect new results soon~\cite{falconouy-2012,espig-als-2013}.

In contrast to ALS schemes which operate on manifolds of fixed dimension, the DMRG algorithm changes the ranks of a tensor format.
This allows to choose the ranks adaptively to the desired error threshold or the accuracy of the result and develop more practical algorithms which do not rely on a priori choice of ranks.
The DMRG was adopted for novel tensor formats (see references above) and new problems, including adaptive high--dimensional interpolation~\cite{so-dmrgi-2011proc} and solution of linear systems~\cite{DoOs-dmrg-solve-2011,holtz-ALS-DMRG-2012}.
The geometrical analysis, eg the convergence of the nonlinear Gauss--Seidel method, is however even more difficult when the dimensions of underlying manifolds are not fixed.

Apart of working with the tensor format structure directly, like ALS and DMRG do, standard algorithms from numerical linear algebra can be applied with tensor approximations and other tensor arithmetics.
Following this paradigm, the solution of linear problems in tensor product formats was addressed in \cite{sav-2006,balgras-Htuck_gmres-2013,dc-tt_gmres-2011}.
The usual considerations of linear algebra can be used in this case to analyze the convergence.
A first notable example is the method of conjugate--gradient type for the Rayleigh quotient minimization in higher dimensions, for which the global convergence was proven by O.~Lebedeva~\cite{lebedeva-tensornd-2011}.

We develop a framework which combines the ALS optimization steps (ranks are fixed, convergence estimates not yet possible) with the steps when the tensor subspaces are increased and the ranks of a tensor format grow.
Choosing the new vectors in accordance with standard linear algebra algorithms, we recast the classical convergence estimates for the proposed algorithm in higher dimensions.
In this paper we consider the case of symmetrical positive definite (SPD) matrices and analyze the convergence in the $A$--norm, i.e. minimize the  \emph{energy function}.
The \emph{basis enrichment} choice follows the steepest descent (SD) algorithm and the convergence of the resulted method is analyzed with respect to the one of steepest descent.
We show that the basis enrichment step combined with the ALS step can be seen as a certain computationally cheap approximation of the DMRG step.
The complexity of the resulted method is equal to the one of ALS and is linear in the mode size and dimension.
Our choice of the basis enrichment appears to be very good for practical computations, and for the considered numerical examples the proposed methods converge almost as fast as the DMRG algorithm.

Summarizing the above, the proposed algorithms have (1) proven geometrical convergence with the estimated rate, (2) practical convergence compared to the one of DMRG, (3) numerical complexity compared to the one of ALS.

\vskip 5mm

The paper is organized as follows.

In Section~\ref{sec:tt} we introduce the tensor train notation and necessary definitions.

In Section~\ref{sec:als} we introduce the basic notation for ALS and DMRG schemes. We also study how the modification of one TT--block affects the ALS problem for its neighbor and describe this in terms of the Galerkin correction method.

In Section~\ref{sec:sd} we develop the family of steepest descent methods for the problems in one, two and many dimensions.
The proposed methods have an inner--outer structure, i.e., a steepest descent step in $d$ dimensions is followed by a steepest descent step in $d-1$ dimension, etc, cf. the interpolation algorithms~\cite{ost-tucker-2008,gos-kryl-2012}.
The convergence of the recursive algorithms in higher dimensions is analyzed using the Galerkin correction framework.
The effect of roundoff/approximation errors is also studied.

Since we make no assumptions on the TT--ranks of the solution, the ranks of the vectors in the proposed algorithms can grow at each iteration and make the algorithm inefficient.
In Section~\ref{sec:prac} we discuss the implementation details, in particular the steps when the tensor approximation is required to reduce the ranks.

In Section~\ref{sec:num} the model numerical experiments demonstrate the efficiency of the method proposed and compare it with other algorithms mentioned in the paper.


\section{Tensor train notation and definitions} \label{sec:tt}
The tensor train (TT) representation of a $d$-dimensional tensor $\x=[x(i_1,\ldots,i_d)]$ is written as the following multilinear map (cf.~\cite{ushmaev-tt-2013})
\begin{equation}\label{eq:tt}
 \begin{split}
  \x = \tau(\bar X) & = \tau(X^{(1)},\ldots,X^{(d)}), \\
  x(i_1,\ldots,i_d) & = X^{(1)}_{\alpha_0,\alpha_1}(i_1) X^{(2)}_{\alpha_1,\alpha_2}(i_2) \ldots X^{(d-1)}_{\alpha_{d-2},\alpha_{d-1}}(i_{d-1}) X^{(d)}_{\alpha_{d-1},\alpha_d}(i_d),
 \end{split}
\end{equation}
where $i_k=1,\ldots,n_k$ are the \emph{mode} (physical) indices, $\alpha_k=1,\ldots,r_k$ are the \emph{rank} indices, $X^{(k)}$ are the tensor train \emph{cores} (TT--cores) and $\bar X = (X^{(1)}, \ldots, X^{(d)})$ denote the whole tensor train.
Here and later we use the Einstein summation convention~\cite{Einstein-relativitat-1916}, which assumes a summation over every pair of repeated indices.
Therefore, in Eq.~\eqref{eq:tt} we assume the summation over all rank indices $\alpha_k,$ $k=1,\ldots,d-1.$
We also imply the \emph{closed boundary conditions}  $r_0=r_d=1$ to make the right--hand side a scalar for each $(i_1,\ldots,i_d).$
Eq.~\eqref{eq:tt} is written in the elementwise form, i.e.,  the equation is assumed over all free (unpaired) indices.
It is often convenient in higher dimensions and will be used throughout the paper.

The indices can be written either in the subscript $x_j$ or in brackets $x(j)$.
For the summation, there is no difference.
The subscripted indices are usually considered as \emph{row and column} indices of a matrix, while the indices in brackets are seen as \emph{parameters}.
For example, each TT--core $X^{(k)}$ is considered as a parameter-dependent on $i_k$ matrix with the row index $\alpha_{k-1}$ and the column index $\alpha_k$ as follows
$$
X^{(k)} = [X^{(k)}_{\alpha_{k-1},\alpha_k}(i_k)] \in \C^{r_{k-1}\times n_k \times r_k}, \qquad X^{(k)}(i_k) \in \C^{r_{k-1}\times r_k}.
$$
In our notation $X^{(k)}(i_k)$ is a matrix, for which standard algorithms like orthogonalization (QR) and singular value decomposition (SVD) can be applied.
We will freely transfer indices from subscripts to brackets in order to make the equations easier to read or to emphasize a certain transposition of elements in tensors.
It brings the notations in consistence with previous papers on the numerical tensor methods, e.g. \cite{holtz-ALS-DMRG-2012,DoOs-dmrg-solve-2011,dk-qtt-tucker-2012pre,ushmaev-tt-2013} and others.

We will reshape arrays into matrices and vectors by using the \emph{index grouping}, i.e., combining two or more indices $\alpha,\ldots,\zeta$ in a single multi-index $\overline{\alpha \ldots \zeta}.$
Following~\cite{ushmaev-tt-2013} we define \emph{interface} matrices $X^{\leq k}\in\C^{n_1\ldots n_k \times r_k}$ and $X^{> k}\in\C^{r_k \times n_{k+1}\ldots n_d}$ as follows
\begin{equation}\label{eq:iface}
 \begin{split}
  X^{\leq k}(\overline{i_1 i_2 \ldots i_k}, \alpha_k) & = X^{(1)}_{\alpha_1}(i_1) X^{(2)}_{\alpha_1\alpha_2}(i_2) \ldots X^{(k)}_{\alpha_{k-1},\alpha_k}(i_k), \\
  X^{> k}(\alpha_k,\overline{i_{k+1}\ldots i_{d-1}i_d}) & = X^{(k+1)}_{\alpha_k,\alpha_{k+1}}(i_{k+1}) \ldots X^{(d-1)}_{\alpha_{d-2},\alpha_{d-1}}(i_{d-1}) X^{(d)}_{\alpha_{d-1}}(i_d),
 \end{split}
\end{equation}
and similarly for symbols $X^{<k}$ and $X^{\geq k}.$
Using the $\tau$ notation defined in~\eqref{eq:tt} we can write $x = \tau(X^{\leq k},X^{> k}).$
For a tensor $\x=[x(i_1,\ldots,x_d)]$ we also define the \emph{unfolding matrix}, which consists of the entries of the original tensor as follows
$$
X^{\{k\}}(\overline{i_1\ldots i_k}, \overline{i_{k+1}\ldots i_d}) = x(i_1,\ldots,i_d),
\qquad
X^{\{k\}}\in \C^{n_1\ldots n_k \times n_{k+1}\ldots n_d}.
$$

For $\x$ in the TT--format~\eqref{eq:tt} it holds $X^{\{k\}} = X^{\leq k} X^{> k}$ and therefore $\rank X^{\{k\}}=r_k.$
In~\cite{osel-tt-2011} the reverse is proven: for any tensor $\x$ there exists the representation~\eqref{eq:tt} with TT--ranks $r_k=\rank X^{\{k\}}.$
This gives the term \emph{TT--rank} the definite algebraic meaning.
As a result, the tensor train representation of fixed TT--ranks yields a closed manifold, and the rank-$(r_1,...,r_{d-1})$ approximation problem is well--posed.
We can also approximate a given tensor by a tensor train with quasi--optimal ranks using a simple and robust approximation (rank truncation, or \emph{tensor rounding}) algorithm~\cite{osel-tt-2011}.
This is the case for all tensor networks without cycles, eg. Tucker~\cite{Tucker}, HT~\cite{hk-ht-2009}, QTT-Tucker~\cite{dk-qtt-tucker-2012pre}, etc.
In contrast, the MPS formalism originally assumes the \emph{periodic boundary conditions} $\alpha_0=\alpha_d$ and sum over these indices, which leads to $\trace(X^{(1)}\ldots X^{(d)}),$ where all matrices can be shifted in cycle under the trace.
The optimization in such type of \emph{tensor networks} is difficult, because they form unclosed manifolds and the best approximation does not always exist.

The tensor train representation of the matrix is made similarly with the TT--cores depending on two parameters $i_k,j_k.$
Hence, $\x=\tau(\bar X)$ is sought in the form~\eqref{eq:tt} and $\A$ and $\y$ given in the TT--format as follows
\begin{equation}\label{eq:matt}
 \begin{split}
  A(i_1,\ldots,i_d;\: j_1,\ldots,j_d) & = A^{(1)}(i_1,j_1) \ldots A^{(d)}(i_d,j_d), \\
  y(i_1,\ldots,i_d) & = Y^{(1)}(i_1) \ldots Y^{(d)}(i_d).
\end{split}
\end{equation}

For $\A$ and $\x$ given in the TT--format, the matrix-vector product $\c = \A \x$ is also a TT--format computed as follows
\begin{equation}\nonumber
 c(i_1,\ldots,i_d) = \left(A^{(1)}(i_1,j_1) \otimes X^{(1)}(j_1)\right) \ldots \left(A^{(d)}(i_d,j_d) \otimes X^{(d)}(j_d)\right),
\end{equation}
where $\otimes$ denotes the tensor (Kronecker) product of two matrices defined as follows
\begin{equation}\nonumber
 A=\begin{bmatrix} A(i,j) \end{bmatrix}, \quad
 B=\begin{bmatrix} B(p,q) \end{bmatrix}, \qquad
 C = A\otimes B = \begin{bmatrix} C(\overline{ip},\overline{jq}) \end{bmatrix}
                = \begin{bmatrix} A(i,j) B(p,q) \end{bmatrix}.
\end{equation}
We refer to \cite{osel-tt-2011} for more details on basic tensor operations in the TT--format.

In this paper we will use standard $l_2$ scalar product $(\cdot, \cdot)$ and the $A$--scalar product $(\cdot, \cdot)_A$ defined by a symmetrical positive definite (SPD) matrix $A$ as follows
\begin{equation}\nonumber
 (u,v)_A = u^\trans A v, \qquad \|u\|_A^2 = (u,u)_A.
\end{equation}
For a given nonsingular matrix $U$ we define the $A$--orthogonal projector $R_U$ as follows: for all $v$ and all $w\in\Span U$ it holds
\begin{equation}\nonumber
  R_U v  \in \Span U, \qquad (w,R_U v)_A = (w,v)_A, \qquad  R_U = U (U^\trans A U)^{-1} U^\trans A.
\end{equation}

We will use vector notations for mode indices $\i=(i_1,\ldots,i_d)$ and rank indices $\r=(r_0,\ldots,r_d).$
We also denote the subspace of tensor trains $\bar X=(X^{(1)},\ldots,X^{(d)})$ with tensor ranks $\r$ as
$$
\T_{\r}  = \mathop{\times}\limits_{i=1}^d \C^{r_{i-1} \times n_i \times r_i}.
$$

\section{Alternating minimization methods}\label{sec:als}
\subsection{ALS--like minimization with fixed TT--ranks}
The MPS formalism was proposed in Quantum Physics, where the representation~\eqref{eq:tt} was used for the minimization of the Rayleigh quotient $(\x,\A\x)/(\x,\x).$
Similarly, the solution of a linear system $\A \x = \y$  with $\A=\A^*$ can be sought through the minimization of an \emph{energy function}
\begin{equation}\label{eq:energy}
 J(\x) = \|\x_* - \x \|_A^2 = (\x,\A\x) - 2\Re(\x,\y) + \const,
\end{equation}
where $\x_*$ denotes the exact solution.
We consider the Hermitian matrix $\A=\A^*$ and the right-hand side $\y$ given in the TT--format~\eqref{eq:matt}, and solve the minimization problem with $\x$ sought in the TT--format~\eqref{eq:tt} with fixed TT--ranks $\r,$ i.e.,
$
  \bar X_* = \arg\min_{\bar X\in\T_\r}J(\tau(\bar X)).
$
This heavy nonlinear minimization problem can hardly be solved unless a (very) accurate initial guess is available (see, eg.~\cite{ushmaev-tt-2013}).
To make it tractable, we can use the alternating linear optimization framework and substitute the global minimization over the tensor train $\bar X \in \T_\r$ by the linear minimization over all cores $X^{(1)}, \ldots, X^{(d)}$ subsequently in a cycle.
Solving the \emph{local} problem we assume that all cores but $k$--th of the current tensor train $\bar X = (X^{(1)},\ldots,X^{(d)})$ are `frozen', and the minimization is done over $X^{(k)}$ as follows
\begin{equation}\label{eq:emin1}
 \begin{split}
  \bar X_\new  & = (X^{(1)},\ldots,X^{(k-1)},X^{(k)}_\new,X^{(k+1)},\ldots,X^{(d)}), \qquad \mbox{where} \\
  \quad X^{(k)}_{\new} & = \arg\min_{X^{(k)}} J(\tau(\bar X)), \qquad \mbox{s.t.} \quad X^{(k)}\in\C^{r_{k-1}\times n_k \times r_k}.
 \end{split}
\end{equation}
Clearly, the energy function does not grow during the sequence of ALS updates and the solution will converge to a local minimum.

To write each ALS step as a linear problem, let us stretch all entries of the TT--core $X^{(k)}$ in the vector
$
x_k(\overline{\alpha_{k-1} i_k \alpha_k}) = X^{(k)}_{\alpha_{k-1},\alpha_k}(i_k).
$
From~\eqref{eq:tt} we see that $\x = \X_{\neq k} x_k,$ where $\X_{\neq k}=\P_{\neq k}(\bar X)$ is the $n_1\ldots n_d \times r_{k-1} n_k r_k$ matrix defined as follows
\begin{equation}\label{eq:proj}
 \begin{split}
 \X_{\neq k}(\overline{i_1\ldots i_d}, \overline{\alpha_{k-1} j_k \alpha_k}) & =  X^{(1)}_{\alpha_1}(i_1) \ldots X^{(k-1)}_{\alpha_{k-2},\alpha_{k-1}}(i_{k-1}) \delta(i_k,j_k)X^{(k+1)}_{\alpha_{k},\alpha_{k+1}}(i_{k+1}) \ldots X^{(d)}_{\alpha_{d-1}}(i_d), \\
 \X_{\neq k} = \P_{\neq k}(\bar X)  & = X^{<k} \otimes I_{n_k} \otimes \left(X^{>k} \right)^\top,
 \end{split}
\end{equation}
where $\delta(i,j)$ is the Kronecker symbol, i.e., $\delta(i,j)=1$ if $i=j$ and $\delta(i,j)=0$ elsewhere.
If $J(\tau(\bar X))$ is considered as a function of $x_k,$ it is also the second-order energy function
$$
J(\tau(\x)) = (\A \X_{\neq k} x_k, \X_{\neq k} x_k) - 2 (\y, \X_{\neq k} x_k) = (\X_{\neq k}^\trans \A \X_{\neq k} x_k, x_k) - 2 (\X_{\neq k}^\trans \y, x_k),
$$
where the gradient w.r.t. $x_k$ is zero when%
\footnote{For illustration see  Fig.~\ref{fig:xax}, for the detailed derivation see~\cite{holtz-ALS-DMRG-2012,DoOs-dmrg-solve-2011}.}
\begin{equation} \label{eq:reduced}
 \left( \X_{\neq k}^\trans \A \X_{\neq k} \right) x_k = \X_{\neq k}^\trans \y.
\end{equation}
The solution of the local minimization problem~\eqref{eq:emin1} is therefore equivalent to the solution of the original system $\A\x=\y$ in the \emph{reduced basis} $\X_{\neq k}=\P_{\neq k}(\bar X),$ defined by~\eqref{eq:proj}.
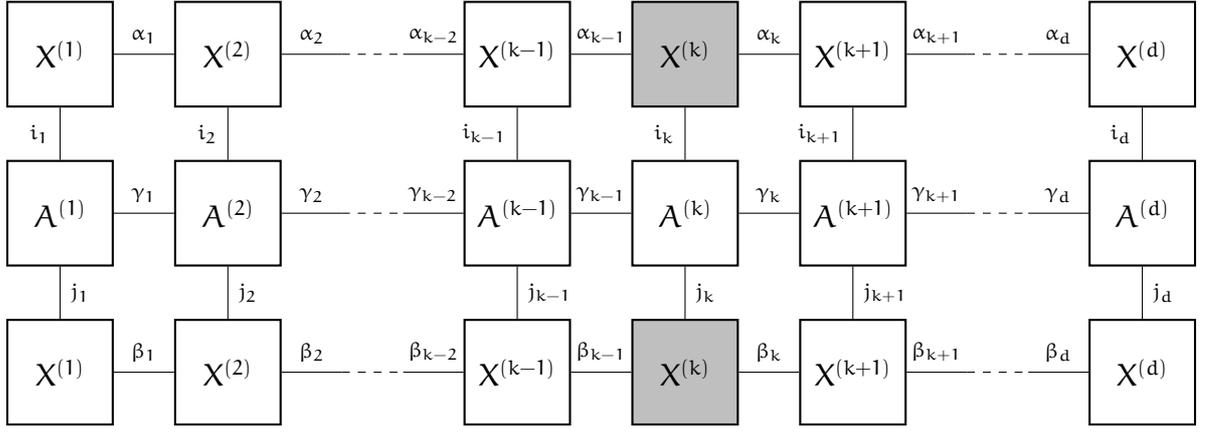
\begin{figure}[t]
 \begin{center}
  \resizebox{.98\textwidth}{!}{ \def\ss{\scriptsize}
 \def\st{\scriptstyle}
 \def\L{\TIKZlspan}
 \def\H{\TIKZhspan}
 \begin{tikzpicture}[x=10mm,y=-10mm]
   \node[core ] (a5)                   {$A^{(k)}$};    
   \node[corex] (x5) [above=\H of a5]  {$X^{(k)}$};
   \node[corex] (y5) [below=\H of a5]  {$X^{(k)}$};
   
   \node[core] (a6) [right=\L of a5]  {$A^{(k+1)}$};
   \node[void] (a7) [right=\L of a6]  {};
   \node[core] (a8) [right=\L of a7]  {$A^{(d)}$};
   \node[core] (a4) [left =\L of a5]  {$A^{(k-1)}$};
   \node[void] (a3) [left =\L of a4]  {};
   \node[core] (a2) [left =\L of a3]  {$A^{(2)}$};
   \node[core] (a1) [left= \L of a2]  {$A^{(1)}$};
   
   \node[core] (x6) [right=\L of x5]  {$X^{(k+1)}$};
   \node[void] (x7) [right=\L of x6]  {};
   \node[core] (x8) [right=\L of x7]  {$X^{(d)}$};
   \node[core] (x4) [left =\L of x5]  {$X^{(k-1)}$};
   \node[void] (x3) [left =\L of x4]  {};
   \node[core] (x2) [left =\L of x3]  {$X^{(2)}$};
   \node[core] (x1) [left= \L of x2]  {$X^{(1)}$};
   
   \node[core] (y6) [right=\L of y5]  {$X^{(k+1)}$};
   \node[void] (y7) [right=\L of y6]  {};
   \node[core] (y8) [right=\L of y7]  {$X^{(d)}$};
   \node[core] (y4) [left =\L of y5]  {$X^{(k-1)}$};
   \node[void] (y3) [left =\L of y4]  {};
   \node[core] (y2) [left =\L of y3]  {$X^{(2)}$};
   \node[core] (y1) [left= \L of y2]  {$X^{(1)}$};

   \path[bond]        (a1) -- node[left,midway]  {${\st i_1}$}     (x1);
   \path[bond]        (a2) -- node[left,midway]  {${\st i_2}$}     (x2);
   \path[bond]        (a4) -- node[left,midway]  {${\st i_{k-1}}$} (x4);
   \path[bond]        (a5) -- node[left,midway]  {${\st i_k}$}     (x5);
   \path[bond]        (a6) -- node[left,midway]  {${\st i_{k+1}}$} (x6);
   \path[bond]        (a8) -- node[left,midway]  {${\st i_d}$}     (x8);
   
   \path[bond]        (a1) -- node[right,midway] {${\st j_1}$}     (y1);
   \path[bond]        (a2) -- node[right,midway] {${\st j_2}$}     (y2);
   \path[bond]        (a4) -- node[right,midway] {${\st j_{k-1}}$} (y4);
   \path[bond]        (a5) -- node[right,midway] {${\st j_k}$}     (y5);
   \path[bond]        (a6) -- node[right,midway] {${\st j_{k+1}}$} (y6);
   \path[bond]        (a8) -- node[right,midway] {${\st j_d}$}     (y8);
   
   \path[bond]        (a1)      -- node[above,midway]  {${\st\gamma_1}$}       (a2);
   \path[bond]        (a2)      -- node[above,midway]  {${\st\gamma_2}$}       (a3.west);
   \path[bond,dashed] (a3.west) --                                             (a3.east);
   \path[bond]        (a3.east) -- node[above,midway]  {${\st\gamma_{k-2}}$}   (a4);
   \path[bond]        (a4)      -- node[above,midway]  {${\st\gamma_{k-1}}$}   (a5);
   \path[bond]        (a5)      -- node[above,midway]  {${\st\gamma_{k}}$}     (a6);
   \path[bond]        (a6)      -- node[above,midway]  {${\st\gamma_{k+1}}$}   (a7.west);
   \path[bond,dashed] (a7.west) --                                             (a7.east);
   \path[bond]        (a7.east) -- node[above,midway]  {${\st\gamma_d}$}       (a8);
   
   \path[bond]        (x1)      -- node[above,midway]  {${\st\alpha_1}$}       (x2);
   \path[bond]        (x2)      -- node[above,midway]  {${\st\alpha_2}$}       (x3.west);
   \path[bond,dashed] (x3.west) --                                             (x3.east);
   \path[bond]        (x3.east) -- node[above,midway]  {${\st\alpha_{k-2}}$}   (x4);
   \path[bond]        (x4)      -- node[above,midway]  {${\st\alpha_{k-1}}$}   (x5);
   \path[bond]        (x5)      -- node[above,midway]  {${\st\alpha_{k}}$}     (x6);
   \path[bond]        (x6)      -- node[above,midway]  {${\st\alpha_{k+1}}$}   (x7.west);
   \path[bond,dashed] (x7.west) --                                             (x7.east);
   \path[bond]        (x7.east) -- node[above,midway]  {${\st\alpha_d}$}       (x8);
   
   \path[bond]        (y1)      -- node[above,midway]  {${\st\beta_1}$}       (y2);
   \path[bond]        (y2)      -- node[above,midway]  {${\st\beta_2}$}       (y3.west);
   \path[bond,dashed] (y3.west) --                                            (y3.east);
   \path[bond]        (y3.east) -- node[above,midway]  {${\st\beta_{k-2}}$}   (y4);
   \path[bond]        (y4)      -- node[above,midway]  {${\st\beta_{k-1}}$}   (y5);
   \path[bond]        (y5)      -- node[above,midway]  {${\st\beta_{k}}$}     (y6);
   \path[bond]        (y6)      -- node[above,midway]  {${\st\beta_{k+1}}$}   (y7.west);
   \path[bond,dashed] (y7.west) --                                            (y7.east);
   \path[bond]        (y7.east) -- node[above,midway]  {${\st\beta_d}$}       (y8);
   
 \end{tikzpicture}
 }
 \end{center}
 \caption{Tensor network corresponding to the quadratic form $(\A\x,\x)$ with matrix $\A$ and vector $\x$ given in the tensor train format.
          The boxes are tensors with lines (legs) denoting indices. Each bond between two tensors assumes a summation over the join index.}
 \label{fig:xax}
\end{figure}

The tensor train representation~\eqref{eq:tt} is non-unique.
Indeed, two representations $\bar X$ and $\bar Y$ map to one tensor $\tau(\bar X)=\tau(\bar Y)$ as soon as
$$
Y^{(k)}(i_k)=H_{k-1}^{-1}X^{(k)}(i_k)H_k, \qquad k=1,\ldots,d,
$$
where $H_0=H_d=1$ and $H_k \in \C^{r_k\times r_k},$ $k=1,\ldots,d-1,$ are arbitrary nonsingular matrices.
Given a vector in the TT--format $\x=\tau(\bar X),$ any transformation $\H=(H_0,\ldots,H_d)$ does not change the energy level since $J(\tau(\bar X))=J(\tau(\bar Y))$ but gives us some flexibility for the choice of the reduced basis since $\P_{\neq k}(\bar X) \neq \P_{\neq k}(\bar Y).$
The proper choice of the \emph{representation} $\bar X$ essentially defines the reduced basis and affects the properties of the \emph{local problem}~\eqref{eq:reduced}.
A prominent transformation $\H$ is the TT--orthogonalization algorithm proposed in~\cite{osel-tt-2011}.
It chooses matrices $H_k$ applying the QR factorization to the reshaped TT--cores, i.e., matrices of size $r_{k-1} \times n_k r_k$ and/or $r_{k-1}n_k \times r_k.$
The transformation $\H$ given by the TT--orthogonalization implies the left--orthogonality constrains on TT--cores $Y^{(1)},\ldots,Y^{(k-1)}$ and right--orthogonality on $Y^{(k+1)},\ldots,Y^{(d)},$ which results in the orthogonality of the interfaces $Y^{< k}$ and $Y^{>k}$ and hence the reduced basis $\Y_{\neq k}=\P_{\neq k}(\bar Y).$
Such a \emph{normalization} step will be  assumed in many algorithms throughout the paper; in most cases we will do this without introduction of a new representation $\bar Y$ just by `claiming' the necessary orthogonalization pattern of the TT representation we use.
If the reduced basis method is applied and such a representation $\bar X$ is chosen so that $\X_{\neq k}=\P_{\neq k}(\bar X)=\P$ is orthogonal, the spectrum of the reduced matrix $\P^\trans A \P$ lies between the minimum and maximum eigenvalues of the matrix $A.$
Indeed, using the Rayleigh quotient~\cite{hornjohnson-1985}, we write
\begin{equation}\nonumber
  \lambda_{\min}(\P^\trans \A \P) = \min_{\|v\|=1} (\P v, A \P v) = \min_{u\in\Span\P, \|u\|=1} (u,Au) \geq \min_{\|u\|=1} (u,Au) = \lmin(\A),
\end{equation}
and similarly for the maximum values. It follows that the reduced matrix is conditioned not worse than the original,
$
\cond(\X_{\neq k}^{\trans} \A \X_{\neq k}) \leq \cond(\A).
$
Therefore, the orthogonality of TT--cores ensures the stability of local problems and we will silently assume this for all reduced problems in this paper.

To conclude this part, let us calculate the complexity of the local problem~\eqref{eq:reduced}.
As was pointed out in \cite{DoOs-dmrg-solve-2011}, either a direct elimination, or an iterative linear solver with fast matrix-by-vector products (\emph{matvecs}) may be applied.
If the direct solution method is used, the costs which are required to form the $r_{k-1}n_kr_k\times r_{k-1} n_k r_k$ matrix of the local problem~\eqref{eq:reduced} are smaller than the complexity of the Gaussian elimination,  i.e., the overall cost is $\O(n^3 r^6).$%
\footnote{We will always assume that $n_1=\ldots=n_d=n$ and $r_1=\ldots=r_{d-1}=r$ in the complexity estimates.}
If an iterative method is used to solve the local problem, one multiplication $\X_{\neq k}^\trans \A \X_{\neq k}$ requires $\O(n r_A r^3 + n^2 r_A^2 r^2)$ operations, where $r$ and $r_A$ denote the TT--rank of the current solution $\x$ and  the matrix $\A$, respectively.
Careful implementation of the matvec is essential to reach this complexity, see~\cite{DoOs-dmrg-solve-2011} for details.
The complexity of the normalization step is only $\mathcal{O}(d n r^3)$ operations and can be neglected.

\subsection{DMRG--like minimization  and adaptivity of TT--ranks}\label{sec:dmrg}
In practical numerical work the TT--ranks of the solution are usually not known in advance, which puts a restriction on the use of the methods with fixed TT--ranks.
The underestimation of TT--ranks leads to a low accuracy of the solution, while the overestimation results in a large computational overhead.
This motivates the development of methods which can choose and modify the TT--ranks on--the--fly adaptively to the desired accuracy level.
A prominent example of such method is the Density Matrix Renormalization Group (DMRG) algorithm~\cite{white-dmrg-1993}, developed in the quantum physics community for the solution of a ground state problem.
DMRG performs similarly to the ALS but at each step combines two succeeding blocks $X^{(k)}$ and $X^{(k+1)}$ into one \emph{superblock}
\begin{equation}\label{eq:super}
w_k(\overline{\alpha_{k-1}i_ki_{k+1}\alpha_{k+1}}) = W^{(k)}_{\alpha_{k-1},\alpha_{k+1}}(\overline{i_ki_{k+1}}), \qquad W^{(k)}(\overline{i_ki_{k+1}}) = X^{(k)}(i_k) X^{(k+1)}(i_{k+1}),
\end{equation}
and make the minimization over $w_k.$
Classical DMRG minimizes the Rayleigh quotient, our version minimizes the energy function $J(\x),$ see~\cite{holtz-ALS-DMRG-2012,DoOs-dmrg-solve-2011}.
Similarly to~\eqref{eq:proj},\eqref{eq:reduced}  we write the local DMRG problem $B w_k = g_k$ as follows
\begin{equation}\label{eq:dmrg}
 \begin{split}
   \P = \P_{\notin \{k,k+1\}}(\bar X) & = X^{<k} \otimes I_{n_k} \otimes I_{n_{k+1}} \otimes \left( X^{>k+1} \right)^\top \in \C^{n_1\ldots n_d \times r_{k-1}n_k n_{k+1}r_{k+1}}, \\
    B & = \P^\trans \A \P, \qquad g_k = \P^\trans \y \in \C^{r_{k-1}n_k n_{k+1}r_{k+1}}.
 \end{split}
\end{equation}
When the $w_k$ is computed, new TT--blocks are obtained by the low--rank decomposition, i.e. the right-hand side of~\eqref{eq:super} is computed and  the $k$-th rank is updated adaptively to the chosen accuracy.
The minimization over $\O(n^2r^2)$ components of $w_k$ leads to complexity $\O(n^3)$, and seriously increases the computational time for systems with large mode sizes.

\subsection{One--block enrichment as a Galerkin reduction of the two-dimensional system} \label{sec:gal}
Suppose that we have just solved~\eqref{eq:reduced} and updated the TT-block $X^{(k)}$.
Before we move to the next step, we would like to improve the reduced basis $\P_{\neq k+1}(\bar X)$ by adding a few vectors to it.
Denote the current solution vector by $\t=\tau(\bar T)$ and suppose we add a step $\s=\tau(\bar S).$
Then the updated solution $\x=\t+\s$ has the  TT--representation $\x=\tau(\bar X)$ defined as follows
\begin{equation}\label{eq:ttsum}
\begin{array}{rcl}
 X^{(1)}(i_1) & := & \begin{bmatrix}T^{(1)}(i_1) & S^{(1)}(i_1)\end{bmatrix}, \\
 X^{(p)}(i_p) & := &\begin{bmatrix}T^{(p)}(i_p) & 0 \\ 0 & S^{(p)}(i_p)\end{bmatrix}, \quad
 X^{(d)}(i_d):=\begin{bmatrix}T^{(d)}(i_d) \\ S^{(d)}(i_d)\end{bmatrix},
\end{array}
\end{equation}
where $p=2,\ldots,d-1.$
We will denote this tensor train as $\bar X = \bar T + \bar S.$%
\footnote{Due to the non--uniqueness of the TT--format other representations (probably with smaller TT--ranks) can exist for $\x=\t+\s.$}
The considered update affects the solution process in two ways: first, naturally, adds a certain correction to the solution, and second, enlarge the reduced basis that we will use at the \emph{next step} of the ALS minimization.
Indeed, it can easily be seen from definition~\eqref{eq:iface} that
\begin{equation}\nonumber
  X^{< k} = \begin{bmatrix} T^{< k} &  S^{< k} \end{bmatrix}, \qquad
  \left(X^{> k}\right)^\top = \begin{bmatrix} \left(T^{> k}\right)^\top &  \left(S^{> k}\right)^\top \end{bmatrix}.
\end{equation}
From~\eqref{eq:proj} we conclude that
$
\P_{\neq k}(\bar T + \bar S) = \begin{bmatrix} T^{< k} &  S^{< k} \end{bmatrix} \otimes I \otimes \begin{bmatrix} \left(T^{> k}\right)^\top &  \left(S^{> k}\right)^\top \end{bmatrix}
$
and hence
\begin{equation}\label{eq:TS}
  \X_{\neq k} = \P_{\neq k}(\bar T + \bar S) = \begin{bmatrix}
            \P_{\neq k}(\bar T)  &
            S^{<k} \otimes I \otimes \left(T^{>k}\right)^\top &
            T^{<k} \otimes I \otimes \left(S^{>k}\right)^\top &
            \P_{\neq k}(\bar S) \end{bmatrix}.
\end{equation}
The clever choice of $\s=\tau(\bar S)$ allows to add the essential vectors to $\Span\P_{\neq k}(\bar T + \bar S)$ and therefore improve the convergence of ALS.

A random choice of $\s\in\T_\r$ with some small TT--ranks $\r$ (cf. \emph{random kick} proposed in~\cite{so-dmrgi-2011proc,Os-mvk2-2011}) may lead to a slow convergence.
It also introduces an unwanted perturbation of the solution.
A more robust idea is to choose $\s$ in accordance to some one-step iterative method, for instance, take $\s \approx \z=\y-\A\t$ and construct a steepest descent or minimal residual method with approximations.
This choice allows to derive the convergence estimate similarly to the classical one and will be discussed in Sec.~\ref{sec:sd}.

To stay within methods of linear complexity, we restrict ourselves to zero shifts $\s=0$ with a simple TT--structure $\bar S=(0,\ldots,0,S^{(k)},0,\ldots,0).$
The tensor train $\bar X = \bar T + \bar S$ has the following structure%
\footnote{We give a description for the \emph{forward} sweep, i.e. the one with increasing $k=1,\ldots,d.$
For the \emph{backward} sweep the construction is done analogously.
}
\begin{equation}\label{eq:rich1}
 X^{(k)}(i_k) = \begin{bmatrix}T^{(k)}(i_k) & S^{(k)}(i_k)\end{bmatrix}, \qquad
 X^{(k+1)}(i_{k+1}) = \begin{bmatrix}T^{(k+1)}(i_{k+1}) \\ 0\end{bmatrix},
\end{equation}
and $X^{(p)}=T^{(p)}$ for other $p.$
Note that since $\s=0,$ the enrichment step does not affect the energy $J(\tau(\bar X)) = J(\tau(\bar T)).$
Therefore, we can choose $S^{(k)}$ freely and develop (probably, heuristic) approaches to improve the convergence of our scheme.
The reduced basis $\X_{\neq k+1} = \P_{\neq k+1}(\bar T + \bar S)$ depends on the choice of $S^{(k)}$ as follows (cf.~\eqref{eq:proj})
\begin{equation}\label{eq:S1}
 \begin{split}
  \X_{\neq k+1}(\overline{i_1\ldots i_d}, \overline{\alpha_k j_{k+1} \alpha_{k+1}}) & =
  X^{<k}(\overline{i_1 \ldots i_{k-1}},\alpha_{k-1}) X^{(k)}_{\alpha_{k-1},\alpha_k}(i_k) \delta(i_{k+1},j_{k+1}) X^{>k+1}(\alpha_{k+1},\overline{i_{k+2}\ldots i_d}),
  \\
  \X_{\neq k+1} & = X^{<k}_{\alpha_{k-1}} \otimes X^{(k)}_{\alpha_{k-1}} \otimes I_{n_{k+1}} \otimes \left( X^{>k+1} \right)^\top,
  \\
  X^{(k)}_{\alpha_{k-1}} & = \begin{bmatrix} T^{(k)}_{\alpha_{k-1}} & S^{(k)}_{\alpha_{k-1}} \end{bmatrix},
 \end{split}
\end{equation}
where $X^{<k}_{\alpha_{k-1}}$ is a column of $X^{<k}$ and $X^{(k)}_{\alpha_{k-1}}$ is the $n_k \times r_k$ matrix which is the \emph{slice} of  3-tensor $X^{(k)}=[X^{(k)}(\alpha_{k-1},i_k,\alpha_k)]$ corresponding to the fixed $\alpha_{k-1},$ and similarly for $S^{k}(\alpha_{k-1})$ and $T^{(k)}(\alpha_{k-1}).$
Below we will write the local system~\eqref{eq:reduced} at the step $k+1$ and see how it is affected by the choice of $S^{(k)}.$

\begin{figure}[t]
 \begin{center}
  \resizebox{.98\textwidth}{!}{ \def\TIKZcoresize{7mm}

 \def\ss{\scriptsize}
 \def\st{\scriptstyle}
 \def\L{\TIKZlspan}
 \def\H{\TIKZhspan}
 \begin{tikzpicture}[x=10mm,y=-10mm]
   \node[core]  (a5)                   {};    
   \node[corex] (x5) [above=\H of a5]  {};
   \node[corez] (x4) [left =\L of x5]  {};
   \node[corev] (y5) [below=\H of a5]  {};
   \node[corez] (y4) [left =\L of y5]  {};
   
   \node[core] (a6) [right=\L of a5]  {};
   \node[void] (a7) [right=\L of a6]  {};
   \node[core] (a8) [right=\L of a7]  {};
   \node[core] (a4) [left =\L of a5]  {};
   \node[core] (a3) [left =\L of a4]  {};
   \node[void] (a2) [left =\L of a3]  {};
   \node[core] (a1) [left= \L of a2]  {};
   
   \node[core] (x6) [right=\L of x5]  {};
   \node[void] (x7) [right=\L of x6]  {};
   \node[core] (x8) [right=\L of x7]  {};
   \node[core] (x3) [left =\L of x4]  {};
   \node[void] (x2) [left =\L of x3]  {};
   \node[core] (x1) [left= \L of x2]  {};
   
   \node[core] (y6) [right=\L of y5]  {};
   \node[void] (y7) [right=\L of y6]  {};
   \node[core] (y8) [right=\L of y7]  {};
   \node[core] (y3) [left =\L of y4]  {};
   \node[void] (y2) [left =\L of y3]  {};
   \node[core] (y1) [left= \L of y2]  {};

   \node[void] (l1) [above=.8\H of x1]  {${\st \phantom{0+}1}$};
   \node[void] (l2) [above=.8\H of x2]  {${\st \ldots}$};
   \node[void] (l3) [above=.8\H of x3]  {${\st k-1}$};
   \node[void] (l4) [above=.8\H of x4]  {${\st k\phantom{+0}}$};
   \node[void] (l5) [above=.8\H of x5]  {${\st k+1}$};
   \node[void] (l6) [above=.8\H of x6]  {${\st k+2}$};
   \node[void] (l7) [above=.8\H of x7]  {${\st \ldots}$};
   \node[void] (l8) [above=.8\H of x8]  {${\st d\phantom{+0}}$};

   \node[void] (lx) [left=.8\L of x1]     {${\st \P_{\neq k+1}(\bar X) x_{k+1}}$};
   \node[void] (la) [left=.8\L of a1]     {${\st A}$};
   \node[void] (ly) [left=.8\L of y1]     {${\st \P_{\neq k+1}(\bar X)}$};

   \path[bond]        (a1) -- node[left,midway]  {}                (x1);
   \path[bond]        (a3) -- node[left,midway]  {}                (x3);
   \path[bond]        (a4) -- node[left,midway]  {${\st j_{k}}$}   (x4);
   \path[bond]        (a5) -- node[left,midway]  {${\st j_{k+1}}$} (x5);
   \path[bond]        (a6) -- node[left,midway]  {}                (x6);
   \path[bond]        (a8) -- node[left,midway]  {}                (x8);
   
   \path[bond]        (a1) -- node[right,midway] {}                (y1);
   \path[bond]        (a3) -- node[right,midway] {}                (y3);
   \path[bond]        (a4) -- node[right,midway] {${\st i_{k}}$}   (y4);
   \path[bond]        (a5) -- node[right,midway] {${\st i_{k+1}}$} (y5);
   \path[bond]        (a6) -- node[right,midway] {}                (y6);
   \path[bond]        (a8) -- node[right,midway] {}                (y8);
   
   \path[bond]        (a1)      -- node[above,midway]  {}         (a2.west);
   \path[bond,dashed] (a2.west) --                                (a2.east);
   \path[bond]        (a2.east) -- node[above,midway]  {}         (a3);
   \path[bond]        (a3)      -- node[above,midway]  {}          (a4);
   \path[bond]        (a4)      -- node[above,midway]  {${\st\gamma_{k}}$}   (a5);
   \path[bond]        (a5)      -- node[above,midway]  {}          (a6);
   \path[bond]        (a6)      -- node[above,midway]  {}          (a7.west);
   \path[bond,dashed] (a7.west) --                                 (a7.east);
   \path[bond]        (a7.east) -- node[above,midway]  {}          (a8);
   
   \path[bond]        (x1)      -- node[above,midway]  {}                     (x2.west);
   \path[bond,dashed] (x2.west) --                                            (x2.east);
   \path[bond]        (x2.east) -- node[above,midway]  {}                     (x3);
   \path[bond]        (x3)      -- node[above,midway]  {${\st\beta_{k-1}}$}  (x4);
   \path[bond]        (x4)      -- node[above,midway]  {${\st\beta_{k}}$}    (x5);
   \path[bond]        (x5)      -- node[above,midway]  {${\st\beta_{k+1}}$}  (x6);
   \path[bond]        (x6)      -- node[above,midway]  {}                     (x7.west);
   \path[bond,dashed] (x7.west) --                                            (x7.east);
   \path[bond]        (x7.east) -- node[above,midway]  {}                     (x8);
   
   \path[bond]        (y1)      -- node[above,midway]  {}                     (y2.west);
   \path[bond,dashed] (y2.west) --                                            (y2.east);
   \path[bond]        (y2.east) -- node[above,midway]  {}                     (y3);
   \path[bond]        (y3)      -- node[above,midway]  {${\st\alpha_{k-1}}$}  (y4);
   \path[bond]        (y4)      -- node[above,midway]  {${\st\alpha_{k}}$}    (y5);
   \path[bond]        (y5)      -- node[above,midway]  {${\st\alpha_{k+1}}$}  (y6);
   \path[bond]        (y6)      -- node[above,midway]  {}                     (y7.west);
   \path[bond,dashed] (y7.west) --                                            (y7.east);
   \path[bond]        (y7.east) -- node[above,midway]  {}                     (y8);
   \node[corev]  (u5)         [below=\H of y5]            {};
   \node[core]   (u6)         [right=\L of u5]            {};
   \node[void]   (u7)         [right=\L of u6]            {};
   \node[core]   (u8)         [right=\L of u7]            {};
   \node[corez]  (u4)         [left=\L  of u5]            {};
   \node[core]   (u3)         [left=\L  of u4]            {};
   \node[void]   (u2)         [left=\L  of u3]            {};
   \node[core]   (u1)         [left=\L  of u2]            {};

   \node[core]   (f5)         [below=\H of u5]            {};
   \node[core]   (f6)         [right=\L of f5]            {};
   \node[void]   (f7)         [right=\L of f6]            {};
   \node[core]   (f8)         [right=\L of f7]            {};
   \node[core]   (f4)         [left=\L  of f5]            {};
   \node[core]   (f3)         [left=\L  of f4]            {};
   \node[void]   (f2)         [left=\L  of f3]            {};
   \node[core]   (f1)         [left=\L  of f2]            {};
   
   \path[bond]        (u1)      -- node[above,midway]  {}                     (u2.west);
   \path[bond,dashed] (u2.west) --                                            (u2.east);
   \path[bond]        (u2.east) -- node[above,midway]  {}                     (u3);
   \path[bond]        (u3)      -- node[above,midway]  {${\st\alpha_{k-1}}$}  (u4);
   \path[bond]        (u4)      -- node[above,midway]  {${\st\alpha_{k}}$}    (u5);
   \path[bond]        (u5)      -- node[above,midway]  {${\st\alpha_{k+1}}$}  (u6);
   \path[bond]        (u6)      -- node[above,midway]  {}                     (u7.west);
   \path[bond,dashed] (u7.west) --                                            (u7.east);
   \path[bond]        (u7.east) -- node[above,midway]  {}                     (u8);

   \path[bond]        (f1)      -- node[above,midway]  {}                     (f2.west);
   \path[bond,dashed] (f2.west) --                                            (f2.east);
   \path[bond]        (f2.east) -- node[above,midway]  {}                     (f3);
   \path[bond]        (f3)      -- node[above,midway]  {}                     (f4);
   \path[bond]        (f4)      -- node[above,midway]  {}                     (f5);
   \path[bond]        (f5)      -- node[above,midway]  {}                     (f6);
   \path[bond]        (f6)      -- node[above,midway]  {}                     (f7.west);
   \path[bond,dashed] (f7.west) --                                            (f7.east);
   \path[bond]        (f7.east) -- node[above,midway]  {}                     (f8);
   
   \path[bond]        (u1) -- node[right,midway] {}                (f1);
   \path[bond]        (u3) -- node[right,midway] {}                (f3);
   \path[bond]        (u4) -- node[right,midway] {${\st i_{k}}$}   (f4);
   \path[bond]        (u5) -- node[right,midway] {${\st i_{k+1}}$} (f5);
   \path[bond]        (u6) -- node[right,midway] {}                (f6);
   \path[bond]        (u8) -- node[right,midway] {}                (f8);
   
   \node[void]  (lu) [left=.8\L of u1]     {${\st \P_{\neq k+1}(\bar X)}$};
   \node[void]  (lf) [left=.8\L of f1]     {${\st y}$};

   \path[draw=white]  (y5) -- node[midway] {{\Large $=$}} (u5);
 \end{tikzpicture}
 }
 \end{center}
 \caption{Linear system $\A\x=\y$ in the reduced basis $\P_{\neq k+1}(\bar X)$ shown by tensor networks.
         The reduced system has $r_{k}n_{k+1}r_{k+1}$ unknowns, shown by the dark box.
         Gray boxes show the $X^{(k)}$ which is updated by $S^{(k)}$ to improve the convergence.
         White boxes contribute to the local matrix $B$ and right-hand side $g$ of the 2D system~\eqref{eq:dmrg}.
         }
 \label{fig:reduced}
\end{figure}
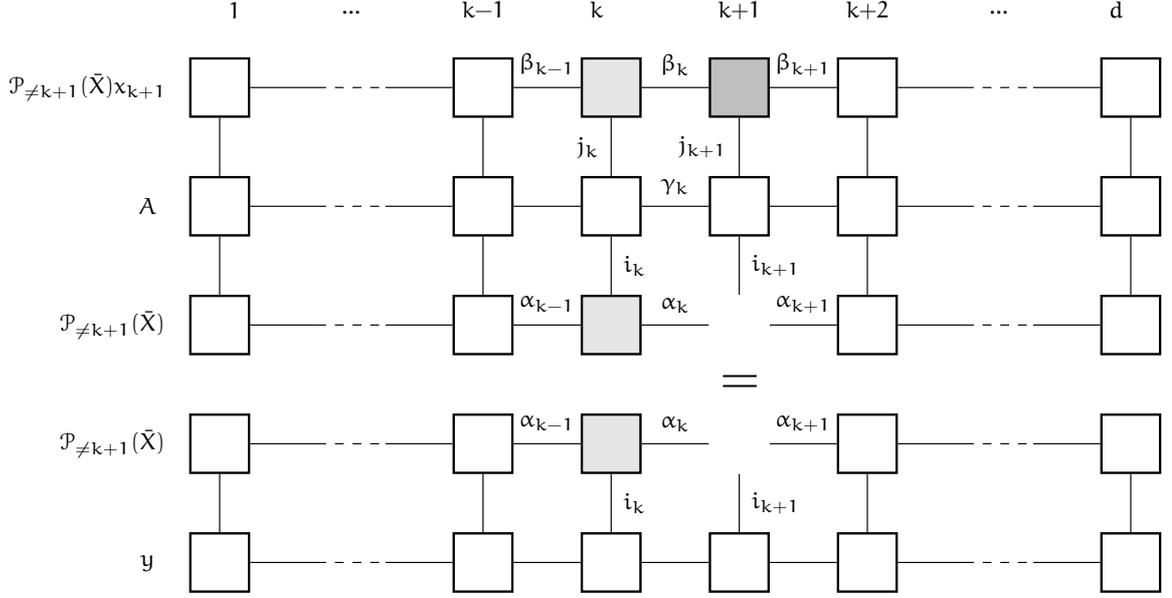
%
The two-dimensional system defined by~\eqref{eq:dmrg} is shown by gray boxes in the Fig.~\ref{fig:reduced}.
It appears here as the local problem in the DMRG method, but in the same framework we may consider the whole initial system with $d=2$, and $k=1$, depending on what type of analysis we would like to perform.

Now the reduced system for the elements of $x_{k+1}(\overline{\beta_k j_{k+1}\beta_{k+1}})=X^{(k+1)}(\beta_k,j_{k+1},\beta_{k+1})$ writes
\begin{equation}\label{eq:Bsys}
 \begin{split}
  X^{(k)}_{\alpha_k,a} B_{ab,a'b'} X^{(k)}_{a',\beta_k} X^{(k+1)}_{\beta_k,b'} & = X^{(k)}_{\alpha_k,a} G_{a,b}, \\
  \begin{bmatrix} X^{(k)} \otimes I \end{bmatrix}^\trans
  B
  \begin{bmatrix} X^{(k)} \otimes I \end{bmatrix} x_{k+1} & =
  \begin{bmatrix} X^{(k)} \otimes I \end{bmatrix}^\trans g,
 \end{split}
\end{equation}
where the following multi-indices are introduced for brevity of notation
\begin{equation}\nonumber
 \begin{array}{lll}
 \overline{\alpha_{k-1} i_k} = a,     & \overline{\beta_{k-1}  j_k} = a',     & a,a' = 1,\ldots,r_{k-1}n_k, \\
 \overline{i_{k+1} \alpha_{k+1}} = b, & \overline{j_{k+1}  \beta_{k+1}} = b', & b,b' = 1,\ldots,r_{k+1}n_{k+1},\\
 \end{array}
\end{equation}
and $X^{(k)} \in \C^{r_{k-1}n_k \times r_k},$  $I = I_{r_{k+1}n_{k+1}}.$
The system~\eqref{eq:Bsys} has $r_kn_{k+1}r_{k+1}$ unknowns.
At the same time it is the reduction of a 2D system $B w = g$ which has $r_{k-1}n_k n_{k+1}r_{k+1}$ unknowns.
Therefore, the choice of the enrichment $S^{(k)}$ (as a part of $X^{(k)}$) can be considered as a cheaper approximation of the 2D system solution.
Taking into account the structure of $X^{(k)}$ from~\eqref{eq:S1} we rewrite~\eqref{eq:Bsys} as follows
\begin{equation}\label{eq:Bsys2}
  \begin{bmatrix} T  &  S \end{bmatrix}^\trans
  B
  \begin{bmatrix} T  &  S \end{bmatrix}
  x_{k+1}  =
  \begin{bmatrix} T  &  S \end{bmatrix}^\trans
  g,
  \qquad
  T = T^{(k)} \otimes I, \quad S = S^{(k)} \otimes I.
\end{equation}

The system \eqref{eq:Bsys2} is difficult to analyze.
However, we may propose a certain approximation to its solution, and estimate the quality of the solution to the whole system \eqref{eq:Bsys2} via the properties of the approximation.
Namely, let us consider the zero--padded TT--core $X^{(k+1)}$ in~\eqref{eq:rich1} as the~\emph{initial guess}, i.e., some information about the solution $x_{k+1}$ that we want to use.
For instance, we can apply the block Gauss--Seidel step, restricting the unknown block to the form
\begin{equation}\label{eq:TV}
 X^{(k+1)}(i_{k+1}) = \begin{bmatrix}T^{(k+1)}(i_{k+1}) \\ V(i_{k+1}) \end{bmatrix}, \qquad
  \begin{array}{rl}
  t(\overline{\alpha_k' i_{k+1} \alpha_{k+1}})  & = T^{(k+1)}_{\alpha_k' \alpha_{k+1}}(i_{k+1}), \\
  v(\overline{\alpha_k'' i_{k+1} \alpha_{k+1}}) & = V_{\alpha_k'' \alpha_{k+1}}(i_{k+1}).
  \end{array}
\end{equation}
Then~\eqref{eq:Bsys2} writes as the following overdetermined system
\begin{equation}\nonumber
  \begin{bmatrix} T^\trans \\  S^\trans \end{bmatrix}
  B
  \begin{bmatrix} T t +  S v\end{bmatrix}
  =
  \begin{bmatrix} T^\trans \\  S^\trans \end{bmatrix}
  g,
\end{equation}
and following the Gauss--Seidel step we solve it considering only the lower part
\begin{equation}\label{eq:gal}
 (S^\trans B S) v = S^\trans (g - B T t).
\end{equation}

Equation~\eqref{eq:gal} is a Galerkin reduction method with the basis $S$ applied to the system $B w = g$ with the initial guess~\eqref{eq:super}, and TT--cores $X^{(k)}$ and $X^{(k+1)}$ defined by~\eqref{eq:rich1}.
After~\eqref{eq:gal} is solved, the updated superblock $W^{(k)}_\new$ writes as follows
\begin{equation}\label{eq:dmrgadd}
 \begin{split}
   X^{(k)}(i_k) & =\begin{bmatrix}T^{(k)}(i_k) & S^{(k)}(i_k)\end{bmatrix}, \qquad
   X^{(k+1)}_\new(i_{k+1}) =\begin{bmatrix}T^{(k+1)}(i_{k+1}) \\ V(i_{k+1})\end{bmatrix},
  \\
  W_\new^{(k)}(\overline{i_ki_{k+1}}) & = T^{(k)}(i_k)T^{(k+1)}(i_{k+1}) + S^{(k)}(i_k) V(i_{k+1})
      \\ & = W^{(k)}(\overline{i_ki_{k+1}})  + S^{(k)}(i_k) V(i_{k+1}),
 \end{split}
\end{equation}
which allows to consider the proposed method as a solver for the 2D system, which performs the~\emph{low--rank correction} for the superblock rather than recompute it from scratch.

Equations \eqref{eq:Bsys} and \eqref{eq:gal} can be considered as certain approximate approaches to the solution of the 2D system \eqref{eq:dmrg}.
Different such approaches can be collected into Table \ref{tab:dmrg}, sorted from the highest to the lowest accuracy.
\begin{table}[t]
\begin{center}
 \begin{tabular}{l|cc|cc|c}
 Method                             & \multicolumn{2}{c|}{$X^{(k)}$} & \multicolumn{2}{c|}{$X^{(k+1)}$}     & Complexity   \\
                                    & $T^{(k)}$       & $S^{(k)}$    & $T^{(k+1)}$  & $V$                   &        \\\hline
 DMRG \eqref{eq:dmrg}               & \multicolumn{2}{c|}{optimize}  & \multicolumn{2}{c|}{optimize}        & $\O(r^3n^3)$\\
 AMEn \eqref{eq:Bsys}               & keep            & choose       & \multicolumn{2}{c|}{optimize}        & $\O(r^2n)$\\
 DMRG correction                    & keep            & optimize     & keep         & optimize              & $\O(\rho^3 r n)$\\
 Galerkin correction \eqref{eq:gal} & keep            & choose       & keep         & optimize              & $\O(\rho r n)$
\end{tabular}
\end{center}
\caption{
Comparison of different solution methods for a two--dimensional system~\eqref{eq:dmrg} with blocks given by~\eqref{eq:dmrgadd}.
We may \emph{keep} the block from the previous iteration, \emph{choose} it arbitrary (eg., using quasi--optimal or heuristic choice) or \emph{optimize} solving the reduced system.
In the complexity estimates, $r$ is typical rank of $\bar X$ and $\rho$ is typical rank of $\bar S.$
}
\label{tab:dmrg}
\end{table}

\section{Steepest descent schemes} \label{sec:sd}
\subsection{Steepest descent with perturbation} \label{sec:sd1}
Given the initial guess $t,$ the \emph{steepest descent} (SD) step minimizes the energy function~\eqref{eq:energy} over vectors $x=t+s\alpha,$ where the step is chosen as follows
\begin{equation}\nonumber
 \begin{split}
  s & = -\grad J(t) = y-At = z, \\
  \alpha & = \arg\min J(t+z\alpha) = \frac{(z,z)}{(z,Az)}.
 \end{split}
\end{equation}
The solution after the SD step satisfies the so-called \emph{Galerkin condition} $(z,y-Ax)=0.$
The progress of the SD step can be analyzed in terms of $A$--norms of errors $c=x_*-t$ and $d=x_*-x$ as follows
$$
 x=t+z \frac{\|z\|^2}{\|z\|_A^2},
 \qquad
 d = c - z\frac{\|z\|^2}{\|z\|_A^2} = (I - R_z) c.
$$
This gives interpretation in terms of projections and proves the monotone decrease of the energy function $J_A(x) = \|d\|_A^2 \leq \|c\|_A^2 = J_A(t).$
To estimate the convergence rate, we write
\begin{equation}\nonumber
 \|d\|_A^2 = (c, (I-R_z)^*A(I-R_z) c) = (c, A(I-R_z)c) = \omega^2_z\|c\|_A^2, \quad \omega^2_z=\frac{(c,(I-R_z)c)_A}{(c,c)_A}.
\end{equation}
The convergence rate $\omega_z$ is therefore a square root of the Rayleigh quotient  for $I-R_z$ in the $A$--scalar product.
It can be bounded using the Kantorovich inequality~\cite{kantorovich-1948} as follows
\begin{equation}\label{eq:sdomega}
 \omega^2_z=1-\frac{(z,z)}{(z,Az)} \frac{(z,z)}{(z,A^{-1}z)} \leq \left( \frac{\lmax-\lmin}{\lmax+\lmin} \right)^2,
\end{equation}
where $\lmax$ and $\lmin$ denote the largest and smallest eigenvalues of $A,$ respectively.

The residual $z=y-At$ of the steepest descent method can not be computed exactly for high--dimensional problems.
Suppose that it is approximated by $\tilde z$ and the perturbed SD step is applied as follows
\begin{equation}\label{eq:sd1}
 x=t+\tilde z \frac{\|\tilde z\|^2}{\|\tilde z\|_A^2}, \qquad
 \tilde d=c - \tilde z \frac{\|\tilde z\|^2}{\|\tilde z\|_A^2} = (I-R_{\tilde z})c + R_{\tilde z} (c - \tilde c),
\end{equation}
where $A\tilde c=\tilde z.$
We further restrict ourselves to the perturbations of the following form
\begin{equation}\label{eq:pert}
 z=\tilde z+\dz,\quad (\tilde z, \dz)=0,
 \qquad
  \|\dz\|_A\leq\eps\|\tilde z\|_A\leq\eps\|z\|_A,
\end{equation}
which will appear naturally in our algorithms for higher dimensions.
For such perturbations the second term vanishes, $R_{\tilde z}(c-\tilde c)=0,$ and the perturbation of the SD step writes through the perturbation of $A$--orthogonal projectors as follows
\begin{equation}\nonumber
 \tilde d - d = -\left( R_{\tilde z} - R_z \right) c.
\end{equation}
A comprehensive overview of the perturbation theory for projections, pseudo--inverses and least square problems can be found in~\cite{steward-1977}.
Rather than adapting their results to the case of $A$--orthogonal projectors, we will develop a more accurate estimate for $\tilde d - d$ using specifically the perturbations~\eqref{eq:pert}.
\begin{theorem}\label{thm:sd1}
 For $\tilde z$ given by~\eqref{eq:pert} the progress of the perturbed SD step~\eqref{eq:sd1} writes as follows
 $$
 \|\tilde d\|_A \leq \omega_{\tilde z} \|c\|_A, \qquad \omega_{\tilde z} = \omega_z + \eps\sqrt{2(1-\omega^2_z)} + \frac{1}{2\sqrt2} \eps^3\cond^2(A),
 $$
 where $\omega_z$ is the progress of the unperturbed SD step given by~\eqref{eq:sdomega}.
\end{theorem}
\begin{proof}
 For $z=\tilde z+\dz$ the following simple identity can be verified from definition
 \begin{equation}\nonumber
  R_z - R_{\tilde z} = \frac{\tilde z\: \dz^\trans}{\|\tilde z\|_A^2}(I-R_z^\trans)A + (I-R_{\tilde z})\frac{\dz\: z^\trans}{\|z\|_A^2} A.
 \end{equation}
 The perturbation of the SD step $\tilde d - d = (R_z - R_{\tilde z})c$ writes
 \begin{equation}\nonumber
   \tilde d - d = \tilde p\frac{\|z\|^2}{\|z\|_A^2}  + \tilde z\frac{(p,z)}{\|\tilde z\|_A^2},
 \end{equation}
 where $p=(I-R_z)\dz$ and $\tilde p=(I-R_{\tilde z})\dz.$
 Obviously, $\|\tilde p\|_A \leq \|\dz\|_A \leq \eps \|z\|_A.$
 To estimate the $A$--norm of the second term, we write
 \begin{equation}\nonumber
  \begin{split}
  (z,p) & = z^\trans (I-R_z)\dz
          = z^\trans \dz - \frac{z^\trans z z^\trans A \dz}{\|z\|_A^2}
          = \|\dz\|^2 - \frac{\|z\|^2}{\|z\|_A^2} (z,\dz)_A
   \\   & = (A^{-1}\dz - \gamma z, \dz)_A,
  \end{split}
 \end{equation}
 where $\gamma={\|z\|^2}/{\|z\|_A^2}.$
 Then $|(p,z)| \leq \|A^{-1}\dz - \gamma z\|_A \|\dz\|_A$ and
 \begin{equation}\nonumber
  \begin{split}
   \|A^{-1}\dz - \gamma z\|_A^2
            & = \|A^{-1}\dz\|_A^2 - 2\gamma (\dz, z) + \gamma^2 \|z\|_A^2
              = \gamma^2\|z\|_A^2 + \|\dz\|_{A^{-1}}^2 - 2\gamma \|\dz\|^2
      \\    &  \leq \gamma^2\|z\|_A^2 + \|\dz\|_{A^{-1}}^2.
  \end{split}
 \end{equation}
 Since $\tilde p$ and $\tilde z$ are $A$--orthogonal, we write
 \begin{equation}\nonumber
   \|\tilde d - d\|_A^2 = \|\tilde p\|^2_A \gamma^2  + \frac{|(p,z)|^2}{\|\tilde z\|_A^2}
    \leq \eps^2 \|z\|_A^2\gamma^2 + \eps^2(\gamma^2\|z\|_A^2 + \|\dz\|_{A^{-1}}^2)
    = \eps^2\left(2\frac{\|z\|^4}{\|z\|_A^2} + \|\dz\|_{A^{-1}}^2\right).
 \end{equation}
 Finally, we estimate
 \begin{equation}\nonumber
  \frac{\|\tilde d - d\|_A}{\|c\|_A}
   \leq \eps\sqrt2 \frac{\|z\|^2}{\|z\|_A\|c\|_A} + \eps\frac{\|\dz\|_{A^{-1}}^2 \|z\|_A}{2\sqrt2 \|z\|^2 \|z\|_{A^{-1}}}
   \leq \eps \sqrt{2(1-\omega^2_z)} + \eps^3 \frac{\cond^2(A)}{2\sqrt2},
 \end{equation}
 where the last inequality is based on
 $$
  \|u\|_{A^{-1}} \leq \lmin^{-1/2} \|u\| \leq \lmin^{-1} \|u\|_A,
  \qquad
  \|u\|_{A^{-1}} \geq \lmax^{-1/2} \|u\| \geq \lmax^{-1} \|u\|_A,
  \qquad
  \cond(A)=\frac\lmax\lmin.
 $$
 Since $\|\tilde d\|_A \leq \|\tilde d - d\|_A + \|d\|_A,$ we obtain the statement of the theorem.
\end{proof}
\begin{remark}
 If $\omega_z<1$ there exists $\eps_*>0$ such that for all $0< \eps < \eps_*$ it holds $\omega_{\tilde z}<1.$
 This critical value $\eps_*(\kappa,\omega)$ is the real positive root of the cubic equation $\omega_{\tilde z}(\eps)=1,$ where $\kappa=\cond(A)$ and $\omega$ act as parameters.
 The minimal value of $\eps_*(\kappa,\omega)$ for $\omega \le (\kappa-1)/(\kappa+1)$ and $\kappa\to\infty$  behaves as
 $
 \eps_* = \kappa^{-1} + \O(\kappa^{-3/2}).
 $
\end{remark}

\subsection{Steepest descent in two dimensions} \label{sec:sd2}
Consider the two--dimensional linear system $Ax=y$ written in the elementwise notation as follows%
\footnote{
We consider $x$ and $y$ as vectors and at the same time as two-dimensional arrays $x=[x(j_1,j_2)]$ and $y=[y(i_1,i_2)]$ with the same entries.
We will switch freely between these representations without change of a notation.
}
\begin{equation}\nonumber
 A(\overline{i_1i_2},\overline{j_1j_2}) x(\overline{j_1j_2}) = y(\overline{i_1i_2}), \qquad i_1,j_1=1,\ldots,n_1, \quad i_2,j_2=1,\ldots,n_2.
\end{equation}
As previously, we assume $A$ and $y$ to be given, and $x$ to be sought in the following low-rank decomposition format
\begin{equation}\label{eq:2d}
 \begin{split}
  A(\overline{i_1i_2},\overline{j_1j_2}) & = A^{(1)}_{\gamma}(i_1,j_1) A^{(2)}_{\gamma}(i_2,j_2), \quad A^{(p)} = [A^{(p)}_{\gamma}(i_p,j_p)] \in \C^{n_p \times n_p \times r_A}, \\
  y(\overline{i_1i_2}) & = y^{(1)}_{\beta}(i_1) y^{(2)}_{\beta}(i_2),  \quad Y^{(p)} = [y^{(p)}_{\beta}(i_p)] \in \C^{n_p \times r_y}, \\
  x(\overline{j_1j_2}) & = x^{(1)}_{\alpha}(j_1) x^{(2)}_{\alpha}(j_2),  \quad X^{(p)} = [x^{(p)}_{\alpha}(j_p)] \in \C^{n_p \times r_x},
 \end{split}
\end{equation}
where $p=1,2.$
Given the initial guess $t$ in the same format, we compute the low--rank approximation of the residual $\tilde z \approx z=y-At$ as follows
\begin{equation}\nonumber
  \tilde z(\overline{i_1i_2})  = z^{(1)}_{\zeta}(i_1) z^{(2)}_{\zeta}(i_2),  \qquad
  Z^{(1)} = [z^{(1)}_{\zeta}(i_1)] \in \C^{n_1 \times r_z}, \quad
  Z^{(2)} = [z^{(2)}_{\zeta}(i_2)] \in \C^{n_2 \times r_z}.
\end{equation}
Following the perturbed SD algorithm, we can write the updated solution $x=t+\tilde z\alpha$ in a form
\begin{equation}\label{eq:sol2}
 x(\overline{j_1j_2}) =
  \begin{bmatrix}T^{(1)}(j_1) &  Z^{(1)}(j_1)\end{bmatrix} \:
  \begin{bmatrix}T^{(2)}(j_2) \\ Z^{(2)}(j_2)\alpha\end{bmatrix},
\end{equation}
and optimize by the step size $\alpha.$
Recalling the considerations from Section \ref{sec:gal}, we can consider more efficient optimization steps listed in Table~\ref{tab:dmrg}.
For example, the solution of DMRG system \eqref{eq:dmrg} corresponds to the exact solution of the considered 2D system.
We will particularly consider the Galerkin correction framework, i.e., will optimize over the bottom block of $X^{(2)},$ denoted as $V$ in~\eqref{eq:dmrgadd}.
This is the cheapest method in Table~\ref{tab:dmrg}, and all other methods have better convergence properties.

In the proposed method we choose the step $x = t + Zv$ where
$$
Z = Z^{(1)} \otimes I_{n_2} \in \C^{n_1n_2 \times r_\zeta n_2}, \quad v(\overline{\zeta i_2})=V_\zeta(i_2),
$$
and without the loss of generality assume the orthogonality of $Z.$
Minimization of the energy function $J(x)$ over $v$ leads to the set of Galerkin conditions $Z^\trans(y-Ax)=0$ and the step writes as follows
\begin{equation}\label{eq:sd2}
  x = t + Zv, \qquad (Z^\trans A Z)v = Z^\trans \tilde z.
\end{equation}
Note that if we restrict ourselves to the perturbations $z = \tilde z + \dz$ such that $Z^\trans \dz = 0,$ it holds
$$
 v = (Z^\trans A Z)^{-1} Z^\trans \tilde z = (Z^\trans A Z)^{-1} Z^\trans  z.
$$
Then the accuracy of the proposed method can be estimated similarly to the standard SD step
$
 d = c-Zv = (I-R_Z)c,
$
 and the progress of this step writes
\begin{equation}\label{eq:sd2omega}
 \|d\|_A^2 = \omega^2_Z \|c\|_A^2, \qquad \omega^2_Z=\frac{(c,(I-R_Z)c)_A}{(c,c)_A}.
\end{equation}
Since $\tilde z\in\Span Z$ it follows that $\omega_Z \leq \omega_{\tilde z},$ i.e., the convergence of the proposed method~\eqref{eq:sd2} is not slower than the one of the perturbed SD step~\eqref{eq:sd1} estimated in Thm.~\ref{thm:sd1}.

\begin{remark}
 When $\Span Z = \C^{n_1n_2}$ we converge in one iteration, i.e. $\omega_Z=0.$
 For large $Z$ s.t. $z\in\Span Z$ we can expect $\omega_Z \ll \omega_z.$
 In general, however, the inequality $\omega_Z \leq \omega_z$ is sharp.
 To show this, consider $Z = \begin{bmatrix} z & s \end{bmatrix}$ with $(z,s)=0.$
 It is easy to show that
 \begin{equation}\nonumber
 1-\omega_Z^2 =  \frac{\|z\|^4 \|s\|_A^2}{\|z\|_{A^{-1}}^2 \left( \|s\|_A^2 \|z\|_A^2 - |(s,z)_A|^2  \right) }, \qquad
 \frac{1-\omega_z^2}{1-\omega_Z^2} = 1 -  \frac{|(s,z)_A|^2}{\|s\|_A^2 \|z\|_A^2} \leq 1,
 \end{equation}
 which proves $\omega_Z \leq \omega_z.$
 However, the ratio can be equal to one when $(s,z)_A=0$ and $(s,z)=0$ simultaneously.
 It can happen, eg. if $s$ is an eigenvector of $A.$
 Similarly, if there is a $k$--dimensional invariant subspace of $A$ which is orthogonal to $z,$ we can form $Z=\begin{bmatrix} z & s_1 \ldots s_k  \end{bmatrix}$ from the basis vectors of this subspace and have the same convergence $\omega_z = \omega_Z$ as the SD step does.
\end{remark}

To find the correction term $v$ we have to solve the reduced linear system size $r_zn_2,$ which writes as follows
\begin{equation} \label{eq:Bcor}
 Bv=g, \qquad B=Z^\trans A Z, \qquad g = Z^\trans z.
\end{equation}
Suppose that $n_2$ is still too large for the system to be solved exactly and we find the approximate solution $v\approx v_* = B^{-1}g.$
The simplest idea is to solve the reduced problem by the standard SD method.
The following theorem estimates the progress of such `lazy' approach.
\begin{theorem}\label{thm:sd2}
 Consider the system $Ax=y$ with the initial guess $t$ and error $c=x_*-t.$
 After one outer step of SD~\eqref{eq:sd2} and one inner step of SD applied to the reduced problem~\eqref{eq:Bcor}, the error $d=x_*-x$ writes as follows
 \begin{equation}\label{eq:sd2a}
   \begin{split}
    d        & = \left( (I-R_Z) + Z (I-Q_g) Z^\trans R_Z \right) c,  \\
             & = \left( (I-R_Z) + (I-R_{\tilde z}) R_Z \right) c,  \\
    \|d\|_A^2& = \left( \omega_Z^2 + (1-\omega_Z)^2\omega_g^2 \right) \|c\|_A^2,
   \end{split}
 \end{equation}
 where $Q_g$ is the $B$--orthogonal projector on $g,$ and $\omega_g \leq \omega_{\tilde z}.$
\end{theorem}
\begin{proof}
If $v$ is the obtained (approximate) solution of~\eqref{eq:Bcor}, the progress of the step~\eqref{eq:sd2} is
\begin{equation} \label{eq:sd2b}
 \begin{split}
  d = x_*-x                 & =  c - Z v = (I-R_Z)c + Z(v_*-v), \\
  \|d\|_A^2 = \|x_*-x\|_A^2 & =  \|(I-R_Z)c\|_A^2 + \|Z(v_*-v)\|_A^2 = \omega_Z^2 \|c\|_A^2 + \|v_*-v\|_B^2,
 \end{split}
\end{equation}
where in the last line we use the $A$--orthogonality of the two terms.
The initial guess for $v$ is zero,  and after one step of the SD applied to~\eqref{eq:Bcor} the error is
$$
 v_*-v = (I-Q_g) (v_*-0) = (I-Q_g) B^{-1} g = (I-Q_g) (Z^\trans A Z)^{-1} Z^\trans z.
$$
The first line of the theorem now follows by the definition of $R_Z.$
To prove the second line it is enough to note that
$$
 Z Q_g Z^\trans = \frac{Zgg^\trans Z^\trans A Z Z^\trans}{g^\trans Z^\trans A Z g} = \frac{\tilde z \tilde z^\trans A Z Z^\trans}{\tilde z^\trans A \tilde z} = R_{\tilde z} Z Z^\trans,
$$
and $ZZ^\trans R_Z = R_Z.$
The progress of the inner SD step is $\|v_*-v\|_B = \omega_g \|v_*\|_B,$ where
\begin{equation} \nonumber
 \|v_*\|_B^2 = \|B^{-1}g\|_B^2
           = \|Z^\trans \tilde z\|_{B^{-1}}^2
           = \|Z^\trans z\|_{B^{-1}}^2
           = (z, Z(Z^\trans A Z)^{-1}Z^\trans z)
           = (c,R_Zc)_A
           = (1-\omega_Z^2) \|c\|_A^2.
\end{equation}
Substituting these estimates to~\eqref{eq:sd2b} we obtain the second claim of the theorem.

Now we prove that $\omega_g \leq \omega_{\tilde z}.$
Similarly to~\eqref{eq:sdomega} we have
$$
 \omega_g^2 = \frac{(v_*,(I-Q_g)v_*)_B}{(v_*,v_*)_B}
            = 1 - \frac{\|g\|^4}{(g,Bg)(g,B^{-1}g)}.
$$
Since $Z$ is orthogonal, $\|g\| = \|Z^\trans z\| = \|\tilde z\|.$
It also holds that $(g,Bg) = (Zg,AZg) = (\tilde z,A \tilde z).$
Finally we show that
$$
(g,B^{-1}g) = (\tilde z,Z (Z^\trans A Z)^{-1} Z^\trans \tilde z)
           = (A^{-1}\tilde z, R_z A^{-1} \tilde z)_A
            \leq \|A^{-1}\tilde z\|_A^2
            = (\tilde z, A^{-1} \tilde z),
$$
which completes the proof.
\end{proof}
The second term of~\eqref{eq:sd2b} can be written also as follows
\begin{equation} \nonumber
 \begin{split}
  Z(v_*-v) & =  Z(I-Q_g)v_*  = Z \left(I - \frac{g g^\trans B}{g^\trans B g} \right) B^{-1}g
             = Z B^{-1} Z^\trans \tilde z - \frac{\|g\|^2}{\|g\|_B^2} \tilde z
     \\ &    = Z B^{-1} Z^\trans z - \frac{\|\tilde z\|^2}{\|\tilde z\|_A^2} z
             = R_Z c - R_{\tilde z} c,
 \end{split}
\end{equation}
which gives $d=(I-R_{\tilde z})c.$
This shows that the combination of one outer and one inner SD step is equivalent to the SD step with perturbation~\eqref{eq:sd1}.
This is also easily seen from the structure of our inner--outer method itself.
Indeed, in the outer step we add components $Z^{(1)}$ to the basis set and in the inner step we add components of the inner residual $g=Z^\trans \tilde z = z_2,$ where $z_2$  contains the elements of $Z^{(2)}$ stretched into one vector.
Therefore, the described inner--outer scheme is equivalent to one `global' SD step.

The idea behind Theorem~\ref{thm:sd2} is of course not to prove a slightly worse estimate in a more complicated way.
In the recursive algorithm the second term in~\eqref{eq:sd2b} will be obtained by the SD step followed by further optimization which will decrease the error of the reduced problem and consequently the total error.
The SD step is therefore required as an initial guess for which we can provide a theoretical estimate of convergence.
The practical convergence that we expect is of course better than the upper estimate in~\eqref{eq:sd2a}.

\begin{remark}
 Regarding the spectrum of reduced problems, the following two--side inequality is proved in~\cite{kantorovich-1990}
 $$
 (U^\trans A U)^{-1} \leq U^\trans A^{-1} U \leq \frac{(\lmin+\lmax)^2}{4\lmin\lmax} (U^\trans A U)^{-1},
 $$
 where $U$ is unitary matrix and $B\geq C$ means that $B-C$ is positive definite.
 The last inequality used in Theorem \ref{thm:sd2} follows from the left  part of this inequality (which is itself rather elementary).
\end{remark}

\subsection{Greedy descent method}\label{sec:talsz}
\begin{algorithm}[t]
 \caption{$x = t + \als(z)$} \label{alg:talsz}
 \begin{algorithmic}[1]
  \REQUIRE System $Ax=y$ and initial guess $t$ in the TT--format~\eqref{eq:tt}, approximate residual $\tilde z=\tau(\bar Z) \in\T_\r.$
  \ENSURE Updated solution $x = t + v,$ $v=\tau(\bar V)\in\T_\r.$
  \FOR[Cycle over TT--cores]{$k=d,\ldots,1$}
    \STATE Find $V^{(k)} = \arg\min_{Z^{(k)}} J(t + \tau(Z^{(1)}, \ldots, Z^{(k-1)}, Z^{(k)}, V^{(k+1)}, \ldots, V^{(d)}))$
  \ENDFOR
  \RETURN $v =  \tau(V^{(1)},\ldots,V^{(d)})$
 \end{algorithmic}
\end{algorithm}
In higher dimensions we can further improve the steepest descent step by an ALS cycle over the step vector, as shown by Alg.~\ref{alg:talsz}.
This algorithm searches for $\max_{s\in\T_\r} J(t+s)$ using the ALS optimization and therefore can be considered as a \emph{greedy} algorithm.
The application of greedy algorithms to optimization in tensor formats was rigorously studied in~\cite{falconouy-2012,Binev-conv_rate_greedy-2011,BrisMaday-greedy_pdes-2009}.

Alg.~\ref{alg:talsz} starts from the SD step with perturbation, and then the energy function is additionally improved by an alternative minimization cycle.
The combined progress is therefore not worse than the one of the SD step, $\|d\|_A\leq\omega_{\tilde z}\|c\|_A,$   given by Thm.~\ref{thm:sd1}.
Another estimate is proven by the following theorem.

\begin{theorem}\label{thm:sdd}
 Consider the system $Ax=y$ with the initial guess $t$ and error $c=x_*-t.$
 The step described by Alg.~\ref{alg:talsz} returns the solution $x=t+v$ such that the error $d=x_*-x$ is bounded as follows
 \begin{equation}\label{eq:sdd}
  \begin{split}
   \|d\|_A^2 & \leq \nu_1^2 \biggl(  \omega_1^2 + (1-\omega_1^2)
                 \nu_2^2 \Bigl(  \omega_2^2 + (1-\omega_2^2)
                 \nu_3^2 \bigl( \omega_3^2 + \ldots
               + \nu_{d-1}^2 \omega_{d-1}^2 \bigr) \ldots \Bigr)\biggr) \|c\|_A^2
     \\ &    = \left( \sum_{k=1}^{d-1} \omega_k^2 \prod_{j=1}^{k-1}(1-\omega_j^2) \prod_{j=1}^k \nu_j^2 \right) \|c\|_A^2,
 \\
   \omega_k^2 & = \omega_{\Z_{\le k}}^2 =  1 -  \frac{( c, R_{\Z_{\le k}} c )_A}{(c,c)_A},
 \qquad
   \nu_k \leq 1.
  \end{split}
 \end{equation}
\end{theorem}
\begin{proof}
In 2D the statement of the theorem reads $\|d\|_A^2 \leq \nu_1^2 \omega_1^2 \|c\|_A^2.$
It is easy to see that the ALS update over $Z^{(2)}$ gives exactly the two--dimensional SD step~\eqref{eq:sd2} with the progress  $\omega_Z = \omega_{\Z_1} = \omega_1$ given by~\eqref{eq:sd2omega}.
The ALS update over $Z^{(1)}$ further improves the energy function by the factor $\nu_1^2 \leq 1,$ which proves the statement of the theorem for $d=2.$
The base of the recursion is proved.

After a \emph{microstep} when $Z^{(k+1)}$ is optimized and becomes $V^{(k+1)}$, the solution writes as follows
\begin{equation}\nonumber
 x_k=t+\Z_{\leq k} v_{>k}, \qquad
      \Z_{\leq k} \in \C^{n_1\ldots n_d \times r_k n_{k+1}\ldots n_d}, \quad
      v_{> k} \in \C^{r_k n_{k+1}\ldots n_d},
\end{equation}
where $v_{>k}=\tau(V^{(k+1)},\ldots,V^{(d)}),$ i.e. $v_{>k}(\overline{\alpha_{k} j_{k+1}\ldots j_d}) = V^{(k+1)}_{\alpha_k \alpha_{k+1}}(j_{k+1})\ldots V^{(d)}_{\alpha_{d-1}}(j_d),$ and
\begin{equation}\label{eq:zproj}
 \begin{split}
  \Z_{\leq k} = \P_{\leq k}(\bar Z) & = Z^{\leq k} \otimes I_{n_{k+1}} \otimes \ldots \otimes I_{n_d}, \\
  \Z_{\leq k}(\overline{i_1 \ldots i_d},\overline{\alpha_k j_{k+1} \ldots j_d})
           & = Z^{(1)}_{\alpha_1}(i_1) Z^{(2)}_{\alpha_1\alpha_2}(i_2) \ldots Z^{(k)}_{\alpha_{k-1}\alpha_k}(i_k) \delta(i_{k+1},j_{k+1}) \ldots \delta(i_d,j_d).
 \end{split}
\end{equation}
This equation is similar to the two--dimensional SD step~\eqref{eq:sd2} and allows to estimate the progress of Alg.~\ref{alg:talsz} using the result of Thm.~\ref{thm:sd2} recursively.
Following~\eqref{eq:sd2b}, the progress can be written as follows
\begin{equation}\label{eq:sdrec}
 \frac{\|x_*-x_k\|_A^2}{\|x_*-t\|_A^2} = \left( \omega_k^2 + (1-\omega_k^2) \frac{\|v_{>k,*}-v_{>k}\|_{A_k}^2}{\|v_{>k,*}-0\|_{A_k}^2} \right),
\end{equation}
where $A_k = \Z_{\leq k}^\trans A \Z_{\leq k},$ $z_k=\Z_{\leq k}^\trans \tilde z$ and $v_{>k,*}$ is the exact solution of the reduced problem $A_k v_{>k} = z_k.$
Note that $z_k=\Z_{\leq k}^\trans \tau(Z^{(1)},\ldots,Z^{(d)}) =\tau(Z^{(k+1)},\ldots,Z^{(d)})$, so the inner SD steps will share the TT--factors of the same residual $\tilde z.$

To prove the recursion step, assume that the theorem holds in the dimension $d-1,$ write~\eqref{eq:sdrec} with $k=1$ and apply~\eqref{eq:sdd} for the second term as follows
\begin{equation}\nonumber
 \frac{\|v_*-v\|_B^2}{\|v_*-0\|_B^2} \le \left( \sum_{k=1}^{d-2} \hat\omega_k^2 \prod_{j=1}^{k-1}(1-\hat\omega_j^2) \prod_{j=1}^k \hat\nu_j^2 \right),
\quad
 \hat\omega_k^2 = 1- \frac{(v_*,Q_{\G_{\leq k}}v_*)_B}{(v_*,v_*)_B},
\end{equation}
where $B=Z^\trans A Z,$ $g=Z^\trans z,$ $Z=\Z_1=Z^{(1)}\otimes I \otimes \ldots \otimes I,$ $v_*$ is the exact solution of $Bv=g,$  $Q_{\G_{\leq k}}$ is the $B$--orthogonal projector on $\G_{\leq k}$ and $\G_{\leq k}=\P_{\leq k}(\bar G)$ is defined for $\tau(\bar G) = g$ similarly to~\eqref{eq:zproj}.
Since $Z \G_{\leq k} = \Z_{\leq k+1}$, and $\|v_*\|_B=\|c\|_A$ we have
$$
(v_*, Q_{\G_{\leq k}} v_*)_B = (z,Z \G_{\leq k} (\G_{\leq k}^\trans B \G_{\leq k})^{-1} \G_{\leq k}^\trans Z^\trans z) = (c, R_{\Z_{\geq k+1}} c)_A,
$$
and  $\hat\omega_k = \omega_{k+1}.$
Similarly $\nu_{k+1} = \hat\nu_k$ now defines the progress of the ALS microstep over the components of $G^{(k)}=Z^{(k+1)}.$
Updating $Z^{(1)}$ by the ALS step we reduce the error by the factor $\nu_1$ and write the total progress as follows
\begin{equation}\nonumber
\frac{\|x_*-x\|_A^2}{\|x_*-t\|_A^2} \le \nu_1^2\left( \omega_1^2 + (1-\omega_1^2) \sum_{k=2}^d \omega_k^2 \prod_{j=2}^{k-1}(1-\omega_j^2) \prod_{j=2}^k \nu_j^2 \right),
\end{equation}
which completes the proof.
\end{proof}

\begin{remark} \label{rem:d-1}
 Under the conditions of the theorem it holds $\|d\|_A\leq\omega_{d-1}\|c\|_A.$
 Indeed, after the first ALS microstep the solution has the form $x_{d-1}=t+\Z_{\leq d-1}v_d,$ see~\eqref{eq:zproj}.
 Comparing this to the steepest descent in 2D~\eqref{eq:sd2} we follow~\eqref{eq:sd2omega} and claim the convergence rate $\omega_{d-1}^2$ for $x_{d-1}$ and consequently for the result of Alg.~\ref{alg:talsz} due to the monotone convergence of the ALS.
\end{remark}

\begin{remark}\label{rem:nu=1}
If ALS steps occasionally give no progress, i.e. $\nu_k=1,$ the progress $\omega$ of Alg.~\ref{alg:talsz} given by~\eqref{eq:sdd} satisfies
$$
1-\omega^2 = (1-\omega_1^2) \ldots (1-\omega_{d-1}^2) = \prod_{k=1}^{d-1} (1-\omega_k^2) \leq 1-\omega_{d-1}^2.
$$
It follows that in this case $\omega^2\geq\omega_{d-1}^2$, and the convergence estimate given by the previous remark is better than the one given by the theorem.
If a sensible estimates for $\nu_k$ are available, we can plug them in~\eqref{eq:sdd} to estimate the combined progress of the SD and ALS steps.
\end{remark}

\subsection{Non-greedy combination of the steepest descent and ALS}\label{sec:alstz}
Alg.~\ref{alg:talsz} is a greedy--type algorithm.
Such algorithms are likely to have a slow convergence or stagnate at some error level.
To improve the practical convergence we can apply the ALS optimization to the whole solution vector $x=t+z\alpha,$ as shown by Alg.~\ref{alg:alstz}.

Just like Alg.~\ref{alg:talsz}, the non-greedy Alg.~\ref{alg:alstz} starts from the steepest descent step and then improves the energy function by a number of ALS updates.
Therefore, the progress of Alg.~\ref{alg:alstz} is estimated by the one of the SD algorithm, $\|d\|_A \leq \omega_{\tilde z}\|c\|_A.$
The better estimate of Remark~\ref{rem:d-1} also applies to Alg.~\ref{alg:alstz}, i.e. $\|d\|_A \leq \omega_{d-1}\|c\|_A.$
This follows from the fact that the optimization over $X^{(d)}$ gives better energy function than the optimization over the lower part of this TT--block $V^{(d)}$, performed in greedy Alg.~\ref{alg:talsz}.
However, we cannot generalize the result of Thm.~\ref{thm:sdd} for Alg.~\ref{alg:alstz}, since the non-greedy ALS update destroys the $\bar T + \bar Z$ structure of the interfaces.
The practically observed convergence of this method is nevertheless much better than that of the greedy descent method.
More rigorous analysis of the convergence of ALS schemes can probably provide much better estimates for the convergence rate of the proposed algorithm.

In the sequel we will develop a version of the algorithm which mixes the ALS and SD steps, following~\eqref{eq:Bsys}, cf. line `AMEn' in Table~\ref{tab:dmrg}.
For this algorithm it is possible to analyze the convergence recurrently similarly to Theorem~\eqref{eq:sdd}.
The mixed AMEn version also has better convergence properties for the practical problems considered in~\cite{dkh-cme-2012}.

\begin{algorithm}[t]
 \caption{$x = \als(t+z)$} \label{alg:alstz}
 \begin{algorithmic}[1]
  \STATE Set $\bar X = (X^{(1)},\ldots,X^{(d)}) = \bar T + \bar Z$
  \FOR[Cycle over TT--cores]{$k=d,\ldots,1$}
    \STATE Find $X^{(k)}_\new = \arg\min_{X^{(k)}} J(\tau(X^{(1)},\ldots,X^{(k)},X^{(k+1)}_\new,\ldots,X^{(d)}_\new))$
  \ENDFOR
  \RETURN $x = \tau(X^{(1)}_\new,\ldots,X^{(d)}_\new)$
 \end{algorithmic}
\end{algorithm}

\section{Practical implementation of tensor truncations} \label{sec:prac}
Throughout the paper, we considered vectors, perturbed due to the tensor approximation.
Now we highlight the practical features of this operation.

The TT--rounding procedure \cite{osel-tt-2011} performs the recursive SVD-based truncations, which reduce the TT--ranks.
The truncation of the $k$-th unfolding writes as follows,
$$
X^{\{k\}}(\overline{i_1 \ldots i_k}, \overline{i_{k+1} \ldots i_d}) = U(\overline{i_1 \ldots i_k},\alpha) \sigma(\alpha) V^{\trans}(\alpha, \overline{i_{k+1} \ldots i_d}),
$$
where matrices $U$ and $V$ are orthogonal.
The approximation algorithm returns
$$
\tilde X^{\{k\}} = \tilde U \tilde U^{\trans} X^{\{k\}}, \quad \delta X^{\{k\}} = (I-\tilde U \tilde U^{\trans}) X^{\{k\}},
$$
where $\tilde U$ contains the $r$ first (dominant) vectors of $U$.
It follows by the construction of the TT--SVD algorithm that $(\tilde X^{\{k\}})^\trans (\delta X^{\{k\}}) = 0$, and therefore
$(\tilde x, \delta x) = (\tau(\tilde X), \tau(\delta X)) = 0$.
We rely on this property for the residual approximation \eqref{eq:pert} in the accuracy analysis of the perturbed steepest descent method, see Theorem \ref{thm:sd1}.
The block version of the same orthogonality condition is used in the derivation of the two-dimensional steepest descent progress \eqref{eq:sd2omega}.

The SVD algorithm truncates a vector in the Frobenius norm, i.e. chooses the approximation rank considering a sum of squared smallest singular values.
To satisfy the accuracy assumption in \eqref{eq:pert} we need to perform the accuracy control in the $A$-norms, $||\delta z||_A \le \eps ||z||_A$.
An optimal approximation in the $A$-norms is a difficult problem.
We can either truncate in the Frobenius norm and rely on the norm equivalence $||x|| \lambda_{\min}^{1/2} \le ||x||_A \le ||x|| \lambda_{\max}^{1/2}$,
or follow the cheap heuristic strategy proposed in \cite{DoOs-dmrg-solve-2011}.
In the inner steps of the TT-rounding procedure, after the SVD is computed, we throw away the smallest singular values one by one, while the \emph{local} error/residual is below the tolerance, i.e.
 $$
 \begin{array}{rcl}
 ||X^{(k)}-U\Sigma V^{\trans}||_{\P_{\neq k}^{\trans} A \P_{\neq k}}  & \le & \eps ||X^{(k)}||_{\P_{\neq k}^{\trans} A \P_{\neq k}}, \quad \mbox{or} \\
 ||\P_{\neq k}^{\trans} A \P_{\neq k}(X^{(k)}-U\Sigma V^{\trans})|| & \le & \eps ||\P_{\neq k}^{\trans} A \P_{\neq k}~X^{(k)}||.
 \end{array}
 $$

The basis enrichment step developed in our paper can only increase the TT--ranks of the solution.
To make the procedure computationally feasible, we need to introduce a truncation step, which will reduce the solution ranks.
To do this, we apply the TT--rounding procedure between the iterations, which perturbs the solution and can increase the energy function.
Therefore, the truncation accuracy has to be chosen accurately to provide the convergence of the methods with approximation.

Assume that a step of the proposed method has the following progress,
$$
\|x_*-x\|_A \le \Omega \|x_*-t\|_A.
$$
The progress after the approximation $\|x - \tilde x\|_A \le \eps_x \|x\|_A$ reads
$$
\|x_*-\tilde x\|_A = \|x_*-x + x-\tilde x\|_A \le \Omega \|x_*-t\|_A + \eps_x \|x\|_A.
$$
While the energy function is large, the first term dominates for sufficiently small $\eps_x$.
In the end of the process, the perturbation error is comparable to the progress of the method, and the algorithm stagnates.
We will see this in numerical examples.

\section{Numerical experiments} \label{sec:num}
Let us verify the methods proposed on a model example of symmetric positive definite system:
$$
-\Delta x = e, \qquad x \in \Omega = [0:1]^d, \qquad \left. x\right|_{\partial\Omega}=0,
$$
where $\Delta$ is the standard finite difference Laplacian discretization on a uniform grid with the mode size $64$ in each direction, i.e., the linear system has $64^d$ unknowns.
The right--hand side $e$ is the vector of all ones.
Such a system arises naturally in the heat transfer simulation, or to precondition more complex elliptic problems.
Note that the matrix and the right--hand side have exact low--rank representations, see~\cite{khkaz-lap-2010,osel-constr-2010}.

For different $d$ we compare the following methods in Fig.~\ref{fig:lp}:
\begin{itemize}
\item the DMRG method presented in \cite{DoOs-dmrg-solve-2011} (``dmrg'');
\item the 2D SD method \eqref{eq:sd2} in a form $x=t+\Z_{\le d-1} v_d$ (``$x=t+Zv$'');
\item the greedy algorithm \ref{alg:talsz} (``$x=t+\als(z)$'');
\item the non-greedy algorithm \ref{alg:alstz} (``$x=\als(t+z)$'');
\item wherever possible, the standard (vectorized) steepest descend (``sd'').
\end{itemize}
The TT--rank of the enrichment $\tilde z$ was chosen $\rho=5$, and the solution after each step was approximated with the relative truncation tolerance $\eps_x = 10^{-4}$ in the Frobenius norm.

\begin{figure}[p]
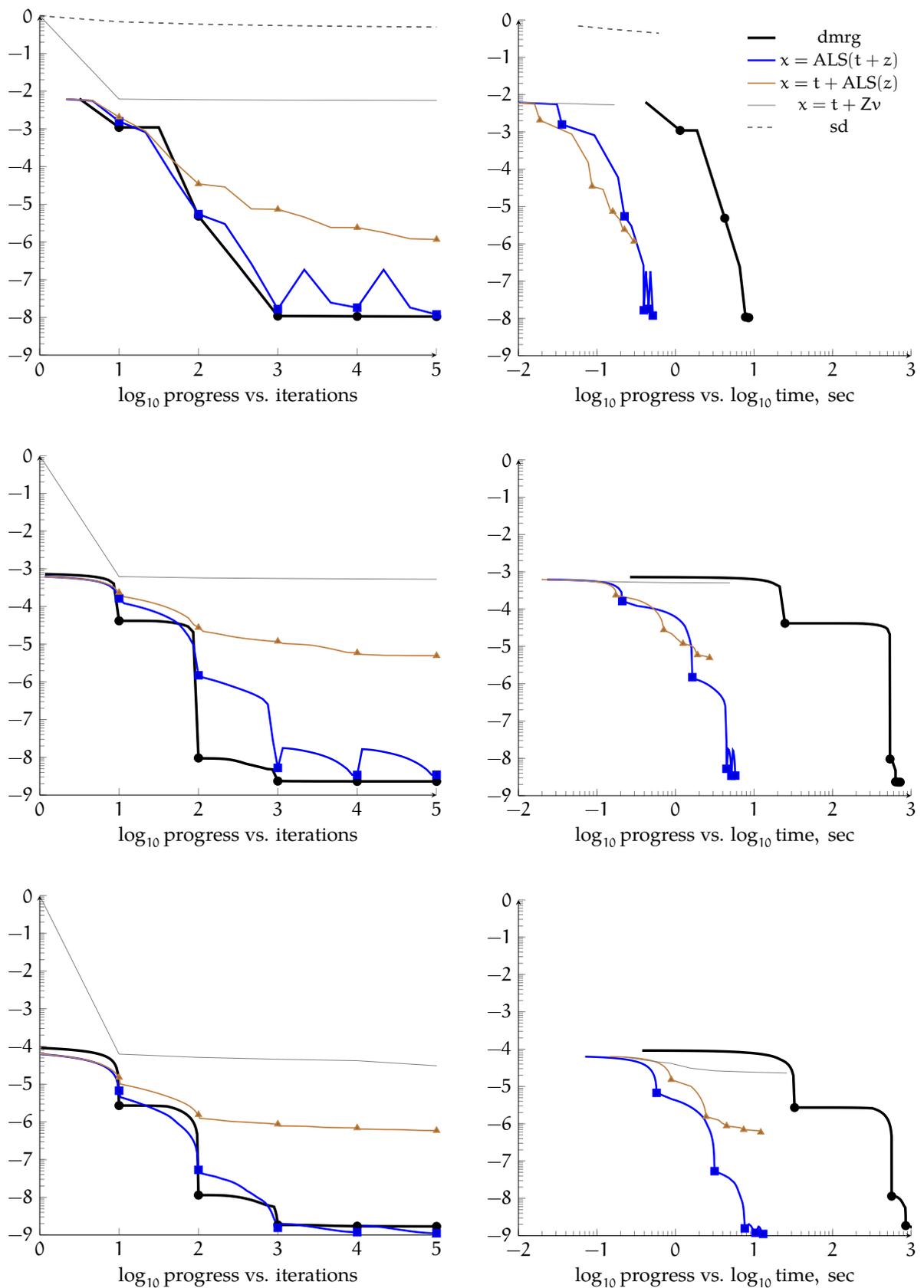

 \begin{center} \hfil  \def\nnn{3}
  \resizebox{.48\textwidth}{!}{\input{./Pic/pgfart.sty} \begin{tikzpicture}

\begin{axis}[%
xmode=normal,ymode=log,
cycle list name=amen,
xlabel={$\log_{10}\mathrm{progress}$ vs. iterations},
xmin=0, xmax=5,
ymin=1e-09, ymax=1,
yminorticks=true,
legend style={at={(.5,-.3)},anchor=north}]

\foreach \m in {dmrg,alsx,xalsv,xv,sd}
  {
  \addplot+[no marks]   table[header=false, x index = 0, y index = 2]{./box/dat/conv_lp\nnn_\m.dat};
  }

\foreach \m in {dmrg,alsx,xalsv}
  {
  \addplot+[only marks] table[header=false, x index = 0, y index = 2]{./box/dat/conv_lp\nnn_\m.int};
  }

\end{axis}
\end{tikzpicture}%
} \hfil
  \resizebox{.48\textwidth}{!}{\input{./Pic/pgfart.sty} \def\lll{1} \begin{tikzpicture}

\begin{axis}[%
xmode=log,ymode=log,
cycle list name=amen,
xlabel={$\log_{10}\mathrm{progress}$ vs. $\log_{10}\mathrm{time,~sec}$ },
xmin=1e-2, xmax=1e3,
xminorticks=true,
ymin=1e-09, ymax=1e0,
yminorticks=true,
legend style={at={(.99,.99)},anchor=north east}
]

\foreach \m in {dmrg,alsx,xalsv,xv,sd}
  {
  \addplot+[no marks]   table[header=false, x index = 1, y index = 2]{./box/dat/conv_lp\nnn_\m.dat};
  }

\foreach \m in {dmrg,alsx,xalsv}
  {
  \addplot+[only marks] table[header=false, x index = 1, y index = 2]{./box/dat/conv_lp\nnn_\m.int};
  }

\ifnum\lll>0
 \legend{dmrg,$x=\als(t+z)$,$x=t+\als(z)$,$x=t+Zv$,sd}
\fi

\end{axis}
\end{tikzpicture}%
} \hfil
 \end{center}
 \begin{center} \hfil  \def\nnn{16}
  \resizebox{.48\textwidth}{!}{\input{./Pic/pgfart.sty} \begin{tikzpicture}

\begin{axis}[%
xmode=normal,ymode=log,
cycle list name=amen,
xlabel={$\log_{10}\mathrm{progress}$ vs. iterations},
xmin=0, xmax=5,
ymin=1e-09, ymax=1,
yminorticks=true,
legend style={at={(.5,-.3)},anchor=north}]

\foreach \m in {dmrg,alsx,xalsv,xv,sd}
  {
  \addplot+[no marks]   table[header=false, x index = 0, y index = 2]{./box/dat/conv_lp\nnn_\m.dat};
  }

\foreach \m in {dmrg,alsx,xalsv}
  {
  \addplot+[only marks] table[header=false, x index = 0, y index = 2]{./box/dat/conv_lp\nnn_\m.int};
  }

\end{axis}
\end{tikzpicture}%
} \hfil
  \resizebox{.48\textwidth}{!}{\input{./Pic/pgfart.sty} \def\lll{0} \begin{tikzpicture}

\begin{axis}[%
xmode=log,ymode=log,
cycle list name=amen,
xlabel={$\log_{10}\mathrm{progress}$ vs. $\log_{10}\mathrm{time,~sec}$ },
xmin=1e-2, xmax=1e3,
xminorticks=true,
ymin=1e-09, ymax=1e0,
yminorticks=true,
legend style={at={(.99,.99)},anchor=north east}
]

\foreach \m in {dmrg,alsx,xalsv,xv,sd}
  {
  \addplot+[no marks]   table[header=false, x index = 1, y index = 2]{./box/dat/conv_lp\nnn_\m.dat};
  }

\foreach \m in {dmrg,alsx,xalsv}
  {
  \addplot+[only marks] table[header=false, x index = 1, y index = 2]{./box/dat/conv_lp\nnn_\m.int};
  }

\ifnum\lll>0
 \legend{dmrg,$x=\als(t+z)$,$x=t+\als(z)$,$x=t+Zv$,sd}
\fi

\end{axis}
\end{tikzpicture}%
} \hfil
 \end{center}
 \begin{center} \hfil  \def\nnn{64}
  \resizebox{.48\textwidth}{!}{\input{./Pic/pgfart.sty} \begin{tikzpicture}

\begin{axis}[%
xmode=normal,ymode=log,
cycle list name=amen,
xlabel={$\log_{10}\mathrm{progress}$ vs. iterations},
xmin=0, xmax=5,
ymin=1e-09, ymax=1,
yminorticks=true,
legend style={at={(.5,-.3)},anchor=north}]

\foreach \m in {dmrg,alsx,xalsv,xv,sd}
  {
  \addplot+[no marks]   table[header=false, x index = 0, y index = 2]{./box/dat/conv_lp\nnn_\m.dat};
  }

\foreach \m in {dmrg,alsx,xalsv}
  {
  \addplot+[only marks] table[header=false, x index = 0, y index = 2]{./box/dat/conv_lp\nnn_\m.int};
  }

\end{axis}
\end{tikzpicture}%
} \hfil
  \resizebox{.48\textwidth}{!}{\input{./Pic/pgfart.sty} \def\lll{0} \begin{tikzpicture}

\begin{axis}[%
xmode=log,ymode=log,
cycle list name=amen,
xlabel={$\log_{10}\mathrm{progress}$ vs. $\log_{10}\mathrm{time,~sec}$ },
xmin=1e-2, xmax=1e3,
xminorticks=true,
ymin=1e-09, ymax=1e0,
yminorticks=true,
legend style={at={(.99,.99)},anchor=north east}
]

\foreach \m in {dmrg,alsx,xalsv,xv,sd}
  {
  \addplot+[no marks]   table[header=false, x index = 1, y index = 2]{./box/dat/conv_lp\nnn_\m.dat};
  }

\foreach \m in {dmrg,alsx,xalsv}
  {
  \addplot+[only marks] table[header=false, x index = 1, y index = 2]{./box/dat/conv_lp\nnn_\m.int};
  }

\ifnum\lll>0
 \legend{dmrg,$x=\als(t+z)$,$x=t+\als(z)$,$x=t+Zv$,sd}
\fi

\end{axis}
\end{tikzpicture}%
} \hfil
 \end{center}
\caption{$A$--norm of the error in different methods versus iterations (left), and CPU time (right).
        Dimension of the problem is $d=3$ (top), $d=16$ (middle), $d=64$ (bottom).}
\label{fig:lp}
\end{figure}

The convergence of the considered methods is compared in Fig.~\ref{fig:lp}.
A one--dimensional sweep is considered as one iteration, the progress of micro-iterations is also shown whenever possible.
We can make the following remarks based on the experimental results.
\begin{itemize}
\item ALS steps sufficiently improves the convergence of all considered methods, i.e. the pessimistic assumptions of Remark~\ref{rem:nu=1} do not hold.
A refined analysis of ALS convergence rates $\nu_k$ is still an open question.
\item The convergence of non-greedy Alg.~\ref{alg:alstz} is comparable to the one of the DMRG iteration-wise.
However, the complexity of each DMRG iteration is cubic in the mode size, while the proposed methods have linear complexity.
This is clearly demonstrated in the right column,  where the convergence is shown w.r.t. the computational time.
The proposed methods  time--wise are up to $100$ times faster than the DMRG for this problem.
\item The one-step steepest descend method shows the slowest convergence, which is a direct consequence of the narrow (one vector)  direction subspace.
 This indicates that the upper bounds of the convergence rate established in the  paper might be seriously overestimated.
\end{itemize}

\section{Conclusion and future work} \label{sec:final}
In this paper we equip the ALS scheme with a basis enrichment step, which is chosen in accordance with the steepest descent algorithm.
The resulted method demonstrates the convergence almost as good as the one of DMRG, while has the linear in the mode size and dimension complexity of ALS.
Moreover, the global convergence rate is established similarly to the one of the steepest descent.
Up to the best of our knowledge, this is the first result on the global convergence of a numerically efficient solution method for linear systems in higher dimensions.
The proposed algorithm combines the advances of optimization methods in tensor formats (ALS, DMRG) with the ones of classical methods of numerical analysis.

The proposed family of methods includes the algorithm with greedy--type step, for which the theoretical results obtained in the framework of greedy algorithms can be applied.
However, other algorithms developed in the non-greedy style also have proven convergence rate and manifest much better convergence in numerical experiments.

The results of this paper can be developed in the following directions.
First, the analysis for the non--symmetric systems can be made similarly to this paper, substituting the steepest descent algorithm by the minimal residual method.
The second Krylov vector is required in MINRES--type algorithms, which have to be approximated and the convergence of perturbed method should be discussed similarly to the Theorem~\ref{thm:sd1}.
Second, the complexity of the proposed methods w.r.t. tensor ranks should be studied and improved using faster (eg, cross) approximation schemes.
Finally, we will develop and analyze the AMEn method for which the enrichment steps are mixed with ALS optimization, i.e., there is no explicit steepest descent step.

The proposed algorithms are already applied to the solution of the chemical master equation in dimensions up to twenty~\cite{dkh-cme-2012}, and more practical applications will follow soon.

\newpage

\end{document}